\documentclass[a4paper,11pt]{article}
\usepackage{makeidx}
\usepackage{graphicx}
\usepackage{yfonts}
\usepackage[english,activeacute]{babel}
\usepackage[latin1]{inputenc}
\usepackage{amsmath}
\usepackage{listings}
\usepackage{amsfonts}
\usepackage{tikz}
\usepackage{caption}
\usepackage[hidelinks]{hyperref}
\usepackage{amssymb}
\usepackage{latexsym}
\usepackage{algpseudocode}
\usepackage{colortbl}
\usepackage{multicol}
\usepackage{bm}
\usepackage{mathrsfs}
\usepackage{orcidlink}
\usepackage{xcolor}
\usepackage{hyperref}

\usetikzlibrary{arrows.meta, positioning}
\definecolor{citecolor}{RGB}{200,0,0} 
\definecolor{linkcolor}{RGB}{0,0,250} 
\hypersetup{
    colorlinks=true,
    linkcolor=linkcolor,
    filecolor=linkcolor,
    urlcolor=linkcolor,
    citecolor=citecolor, }
\newtheorem{theorem}{Theorem}

\newtheorem{corollary}{Corollary}

\newtheorem{definition}{Definition}

\newtheorem{proposition}{Proposition}
\newtheorem{remark}{Remark}

\definecolor{codegreen}{rgb}{0,0.6,0}
\definecolor{codegray}{rgb}{0.5,0.5,0.5}
\definecolor{codepurple}{rgb}{0.58,0,0.82}
\definecolor{backcolour}{rgb}{0.95,0.95,0.92}
\lstdefinestyle{mystyle}{,
    commentstyle=\color{codegreen},
    keywordstyle=\color{magenta},
    numberstyle=\tiny\color{codegray},
    stringstyle=\color{codepurple},
    basicstyle=\ttfamily\footnotesize,
    breakatwhitespace=false,
    breaklines=true,
    captionpos=b,
    keepspaces=true,
    numbers=left,
    numbersep=5pt,
    showspaces=false,
    showstringspaces=false,
    showtabs=false,
    tabsize=2
}
\newenvironment{proof}[1][Proof]{\noindent\textbf{#1.} }{\ \rule{0.5em}{0.5em}}
\begin{document}

\title{Symmetric Truncated Freud polynomials}
\author{{Edmundo J. Huertas$^{1}$\orcidlink{0000-0001-6802-3303}}, {Alberto Lastra$^{1}$\orcidlink{0000-0002-4012-6471}}, {Francisco Marcell\'an$^{2}$\orcidlink{0000-0003-4331-4475}}, {V\'ictor Soto-Larrosa$^{1,3}$\orcidlink{0000-0002-7079-3646}} \\
%EndAName
\\
$^{1}$Departamento de F\'isica y Matem\'aticas, Universidad de Alcal\'a\\
Ctra. Madrid-Barcelona, Km. 33,600\\
28805 - Alcal\'a de Henares, Madrid, Spain\\
edmundo.huertas@uah.es, alberto.lastra@uah.es, v.soto@uah.es\\
\\
$^{2}$Departamento de Matem\'aticas, Universidad Carlos III de Madrid\\
Escuela Polit\'ecnica Superior\\
Av. Universidad 30 28911 Legan\'es, Spain\\
pacomarc@ing.uc3m.es \\
\\
$^{3}$Departamento de Biociencias, Universidad Europea de Madrid\\
Facultad de Ciencias Biom\'edicas y de la Salud\\
 C. Tajo, s/n, 28670 Villaviciosa de Od\'on, Madrid\\
victor.soto@universidadeuropea.es \\
\\
Corresponding author. Email: v.soto@uah.es}
\date{}
\maketitle

\begin{abstract}
We define the family of symmetric truncated Freud polynomials $P_n(x;z)$,
orthogonal with respect to the linear functional $\bm{u}$ defined by 
\begin{equation*}
\langle \bm{u}, p(x)\rangle = \displaystyle\int_{-z}^z p(x)e^{-x^4}dx, \quad
p\in \mathbb{P}, \quad z>0.
\end{equation*}
The semiclassical character of $P_n (x; z)$ as polynomials of class $4$ is
stated. As a consequence, several properties of $P_n (x; z)$ concerning the
coefficients $\gamma_n (z)$ in the three-term recurrence relation they
satisfy as well as the moments and the Stieltjes function of $\bm{u}$ are
studied. Ladder operators associated with such a linear functional and the
holonomic equation that the polynomials $P_n (x; z)$ satisfy are deduced.
Finally, an electrostatic interpretation of the zeros of such polynomials
and their dynamics in terms of the parameter $z$ are given.
\end{abstract}

\section{Introduction}

In the theory of orthogonal polynomials with respect to linear functionals,
a relevant family is constituted by the so-called semiclassical linear
functionals (\cite{Maroni85, Maroni87}). The concept of class allows to
introduce a hierarchy of linear functionals that constitutes an alternative
to the Askey tableau based on the hypergeometric character of the
corresponding sequence of orthogonal polynomials. Semiclassical orthogonal
polynomials appear in the seminal paper \cite{Shohat}, where weight
functions whose logarithmic derivatives are rational functions were
considered. A second order linear differential equation that the
corresponding sequence of orthogonal polynomials is obtained therein. The
theory of semiclassical orthogonal polynomials has been developed by P.
Maroni (see \cite{Maroni91}) and combine techniques of functional analysis,
distribution theory, $z$ and Fourier transforms, ordinary differential
equations, among others, in order to deduce algebraic and structural
properties of such polynomials. Many of the most popular families of
orthogonal polynomials coming from Mathematical Physics as Hermite,
Laguerre, Jacobi and Bessel are semiclassical (of class $s=0$) but other
families appearing in the literature are also semiclassical. For instance,
semiclassical families of class $s=1$ are described in \cite{Belmehdi92}. As
a sample in \cite{Diego} truncated Hermite polynomials associated with the
normal distribution supported on a symmetric interval of the real line have
been considered. They are semiclassical of class $s=2.$\newline

The coefficients of the three-term relation satisfied by orthogonal
polynomials associated with a semiclassical linear functional satisfy
coupled non linear difference equations (Laguerre-Freud equations or string
equations, see \cite{Belmehdi-Ronveaux}), which are related to discrete
Painlev\'e equations see \cite{Magnus99}, \cite{VAssche2018}, among others.
On the other hand, ladder operators (lowering and raising) are associated
with semiclassical orthogonal polynomials. Their construction is based on
the so called first structure relation, see \cite{Maroni91}, they satisfy.
As a consequence, you can deduce a second order linear differential equation
that the polynomials satisfy. Notice that the coefficients are polynomials
with degrees depending on the class.\newline

In the framework of random matrices, several authors have considered
Gaussian unitary ensembles with one or two discontinuities (see \cite{Lyu}).
Therein it is proved that the logarithmic derivative of the Hankel
determinants generated by the normal (Gaussian) weight with two jump
discontinuities in the scaled variables tends to the Hamiltonian of a
coupled Painlev\'e II system and it satisfies a second order PDE. The
asymptotics of the coefficients of the three-term recurrence relation for
the corresponding orthogonal polynomials is deduced. They are related to the
solutions of a coupled Painlev\'e II system. The techniques are based on the
analysis of ladder operators associated with such orthogonal polynomials.
For more information, see \cite{ChenIsmail97, ChenPruessner05}.\newline

Integrable systems also provide nice examples of semiclassical orthogonal
polynomials. They are related to time perturbations of a weight function
(Toda, Langmuir lattices). For more information, see \cite{WVAssche22}.%
\newline

The theory of orthogonal polynomials on unbounded intervals with respect to
linear functionals of the form

\begin{equation}
\int_{-\infty }^{\infty }p(x)|x|^{\rho }e^{-|x|^{m}}dx,\quad p\in \mathbb{P}%
,\quad \rho >-1,\quad m>0,  \label{eq:generalizedFreudLF}
\end{equation}%
started with Géza Freud in the 1970's (\cite{Freud76,Nevai86}). Indeed,
Freud studied the asymptotic behaviour of the recurrence coefficients of
orthogonal polynomials with respect to (\ref{eq:generalizedFreudLF}) for the
cases $m=2,4,6,$ and the corresponding Laguerre-Freud equations for the
recurrence coefficients \cite{Freud76}. Some of them were first obtained in
1885 by Laguerre \cite{Laguerre85} and in 1939 by Shohat \cite{Shohat}. They
are discrete analogues of the Painlevé equations \cite{Magnus99}.\newline
For the general case $m>0$, Freud conjectured the asymptotic behaviour of
the recurrence coefficients, which was proved in 1984 by Magnus \cite%
{Magnus85} when $m$ is an even integer. A more general proof of Freud's
conjecture was given in 1986 by Lubinsky, Mhaskar and Saff \cite{Lubinsky}
when $m>0$.\newline

In particular, if $(P^F_n(x))_{n\geq 0}$ is the sequence of monic
polynomials orthogonal with respect to the linear functional $\bm{u}_F$
defined by

\begin{equation}
\bigl<\bm{u}_F,p(x)\bigr> = \displaystyle\int_{-\infty}^\infty p(x)
e^{-x^4}, \quad p\in \mathbb{P},  \label{eq:standardfreudLF}
\end{equation}
then you have a semiclassical linear functional of class $s=2.$ The
corresponding sequence of orthogonal polynomials $P^F_n(x)$ satisfies a
three-term recurrence relation

\begin{equation*}
xP^F_{n}(x) = P^F_{n+1}(x) +\gamma^F_nP^F_{n-1}(x), \quad n \geq 0.
\end{equation*}
The Laguerre-Freud equation for the coefficients $\gamma^F_{n}$ reduces to
(see \cite{Bauldry88,Freud86,Kriecherbauer1999,Nevai83,Nevai86} among others)

\begin{equation}
4\gamma^F_n(\gamma^F_{n+1}+\gamma^F_n +\gamma^F_{n-1}) = n,\quad n\geq 1.
\label{eq:gammafreud}
\end{equation}

In \cite{Nevai83} P. Nevai proved that (\ref{eq:gammafreud}) has a unique
positive solution with initial conditions

\begin{equation*}
\gamma^F_0 =0, \quad \gamma^F_{1} = \frac{\bigl< \bm{u}_F,x^2\bigr>}{\bigl< %
\bm{u}_F,1\bigr>} = \frac{\Gamma(3/4)}{\Gamma(1/4)}.  \label{eq:IDF}
\end{equation*}
For a more general situation regarding the existence and uniqueness of
positive solutions of nonlinear recurrence relations related with Freud's
weights, see \cite{Alsulami2015}. \newline

In this contribution, we deal with \textit{symmetric truncated Freud
polynomials} $P_{n}(x;z)$, orthogonal with respect to the linear functional $%
\bm{u}$ defined by 
\begin{equation}
\bigl<\bm{u},p(x)\bigr>=\int_{-z}^{z}p(x)e^{-x^{4}}dx,\quad p\in \mathbb{P}%
,\quad z>0,  \label{eq:FreudLF}
\end{equation}%
and satisfying the three-term recurrence relation 
\begin{equation}
xP_{n}(x;z)=P_{n+1}(x;z)+\gamma _{n}(z)P_{n-1}(x;z),\quad n\geq 0.
\label{eq:3trrFreud}
\end{equation}%
Our choice of (\ref{eq:FreudLF}) is motivated by the observation that as $%
z\rightarrow \infty $, we get the Freud linear functional (\ref%
{eq:standardfreudLF}). We use the fact that the linear functional $\bm{u}$
in (\ref{eq:FreudLF}) is semiclassical. This approach provides a natural
framework for exploring the algebraic and analytic properties of the
coefficients involved in the three-term recurrence relation (\ref%
{eq:3trrFreud}). Semiclassical linear functionals were introduced by Pascal
Maroni in \cite{Maroni85,Maroni87}, with a comprehensive overview provided
in \cite{Maroni91}. Notice that the concept of the class of a semiclassical
linear functional enables the establishment of a hierarchy, wherein
classical functionals (such as Hermite, Laguerre, Jacobi, and Bessel) occupy
the bottom. Further details regarding the description of semiclassical
linear functionals of class one can be found in \cite{Belmehdi92}.\newline

The structure of the manuscript is as follows. In Section 2 we provide the
basic background concerning linear functionals and orthogonal polynomials,
focusing particularly on the symmetric and $D$-semiclassical cases. In
Section 3 we study some properties that the linear functional (\ref%
{eq:FreudLF}) satisfies, especially the Pearson equation as well as the
behaviour of the moments and the corresponding Stieltjes function. Section 4
is focused on the nonlinear relations (Laguerre-Freud equations) that the
coefficients $\gamma _{n}(z)$ in (\ref{eq:3trrFreud}) satisfy. In Section 5
we derive a structure relation and the lowering and raising operators
associated with such a linear functional which yield a second order linear
differential (holonomic) equation in $x$ for the polynomials $P_{n}(x;z)$.
In Section 6 we focus our attention on the electrostatic interpretation of
the zeros of the above orthogonal polynomials while their dynamical
behaviour in terms of the parameter $z$ is studied in Section 7. Finally, in
Section 8, some numerical examples regarding the zero behaviour of the
sequence $(P_{n}(x;z))_{n\geq 0}$ are illustrated.

\section{Background}

Let $\mathbb{N}=\{1,2,...\}$ and $\mathbb{N}_0 = \mathbb{N} \cup \{0\}$. Let 
$\mathbb{P}$ be the linear space of polynomials with complex coefficients.
For any linear functional $\bm{\ell}$ and polynomials $h,q\in\mathbb{P}$ let 
$D\bm{\ell}$ and $h(x)\bm{\ell}$ be the linear functionals respectively
defined by

\begin{enumerate}
\item[$i.$] $\bigl< D\bm{\ell},q(x)\bigr> := -\bigl< \bm{\ell},q^{\prime }(x)%
\bigr>$.

\item[$ii.$] $\bigl<h(x)\bm{\ell},q(x)\bigr>:=\bigl<\bm{\ell},h(x)q(x)\bigr>$%
.
\end{enumerate}

In particular, we denote by $\bm{\ell}_n:=\bigl< \bm{\ell},x^n\bigr> $, $%
n\geq 0$, the $n$-th moment of $\bm{\ell}$.\newline

\begin{definition}
A polynomial sequence $(P_n)_{n\geq 0}$ such that $deg (P_n) = n$ is said to
be a \textbf{monic} orthogonal polynomial sequence (MOPS, in short) with
respect to the linear functional $\bm{\ell}$ if

\begin{itemize}
\item[i.] The leading coefficient of $P_n$ is $1$.

\item[ii.] $\bigl<\bm{\ell},P_{n}P_{m}\bigr>=h_{n}\delta _{n,m},\quad n,m\in 
\mathbb{N}_{0},\quad h_{n}\neq 0,$\textup{\label{eq:OPS}}
\end{itemize}

%     \begin{equation}
%         \bigl< \bm{v}, P_nP_m \bigr> = h_n\delta_{n,k}, \quad n,m\in \mathbb{N}_0, \quad h_n\neq 0,
%         \label{eq:OPS}
%     \end{equation}
\end{definition}

where $\delta _{n,m}$ is the Kronecker delta. From (\ref{eq:OPS}) we see
that 
\begin{equation*}
\bigl<\bm{\ell},xP_{n}P_{m}\bigr>=0,\quad 0\leq m<n-1,
\end{equation*}%
and therefore the polynomials $P_{n}$ satisfy a \textit{three-term
recurrence relation} (TTRR, in short) \cite{Chihara78}%
\begin{equation}
xP_{n}(x)=P_{n+1}(x)+\beta _{n}P_{n}(x)+\gamma _{n}P_{n-1}(x),\quad n\geq 1,
\label{eq:3trr}
\end{equation}%
with initial values $P_{0}(x)=1$ and $P_{1}(x)=x-\beta _{0}$. From the
orthogonality condition (\ref{eq:OPS}) the coefficients $\beta _{n}$ and $%
\gamma _{n}$ are given by%
\begin{equation*}
\beta _{n}=\frac{\bigl<\bm{\ell},xP_{n}^{2}\bigr>}{h_{n}},\quad \gamma _{n}=%
\frac{\bigl<\bm{\ell},xP_{n}P_{n-1}\bigr>}{h_{n-1}},\quad n\geq 1,
\end{equation*}%
with%
\begin{equation*}
\beta _{0}=\frac{\bigl<\bm{\ell},x\bigr>}{\bigl<\bm{\ell},1\bigr>}=\frac{%
\bm{\ell}_{1}}{\bm{\ell}_{0}}.
\end{equation*}%
Taking into account that $P_{n}(x)=x^{n}+\mathcal{O}(x^{n-1})$ and using (%
\ref{eq:OPS}) we see that 
\begin{equation*}
h_{n}=\bigl<\bm{\ell},x^{n}P_{n}\bigr>=\bigl<\bm{\ell},xP_{n}P_{n-1}\bigr>%
=\gamma _{n}h_{n-1},
\end{equation*}%
and then 
\begin{equation}
\gamma _{n}=\frac{h_{n}}{h_{n-1}},\quad n\geq 1.  \label{eq:gammah_m}
\end{equation}%
The fact that $\gamma _{n}\neq 0$ for all $n\in \mathbb{N}$ is equivalent to
the nonsingularity of the Hankel matrices $\mathcal{H}_{n}=\bigl(\bm{\ell}%
_{i+j}\bigr)_{i,j\geq 0}^{n},n\geq 0,$ associated with the linear functional 
$\bm{\ell}$. The linear functional $\bm{\ell}$ is said to be quasi-definite
(resp. positive-definite) if every leading principal submatrix of the Hankel
matrix is nonsingular (resp. positive-definite).

\begin{definition}
({\cite{ArdilaMarcellan21, Maroni91}}) A quasi-definite linear functional $%
\bm{\ell}$ is semiclassical of class $\mathfrak{s}$ if and only if the
following statements hold

\begin{enumerate}
\item[i.] There exist non-zero polynomials $\phi(x)$ and $\psi(x)$ with $%
deg(\phi) \geq 0$ and $deg(\psi)\geq 1$ such that $\bm{\ell}$ satisfies the
distributional Pearson equation

\begin{equation}
D(\phi(x;z)\bm{\ell})+\psi(x;z)\bm{\ell}=0.  \label{eq:Pearson}
\end{equation}

\item[ii.] $\displaystyle\prod_{c: \phi(c) = 0}\bigl(\bigl|%
\Psi(c)+\phi^{\prime }(x)\bigr|+\bigl|\bigl< \bm{\ell},\theta_c\Psi(x)+%
\theta^2_c\phi(x)\bigr>\bigr|\bigr) >0,$ where $\theta_c f(x)=\frac{f(x)-f(c)%
}{x-c} $.
\end{enumerate}

Then the class of $\bm{\ell}$ is%
\begin{equation*}
\mathfrak{s}(\bm{\ell})=:\text{min max }\{\deg (\phi )-2,\deg (\psi )-1\},
\end{equation*}%
where the minimum is taken among all pairs of polynomials $\phi (x)$ and $%
\psi (x)$ so that (\ref{eq:Pearson}) holds.
\end{definition}

\begin{definition}
A linear functional $\bm{s}$ is said to be symmetric if all odd moments
vanish, i. e.,%
\begin{equation*}
\bigl<\bm{s},x^{2j+1}\bigr>=0,\quad j\geq 0.
\end{equation*}
\end{definition}

Symmetric functionals are characterized as follows

\begin{theorem}
({\cite[Theorem 4.3]{Chihara78}}) Let $\bigl(S_{n}\bigr)_{n\geq 0}$ be the
sequence of monic orthogonal polynomials with respect to the functional $%
\bm{s}$. Then, the following statements are equivalent

\begin{enumerate}
\item[i.] $\bm{s}$ is symmetric

\item[ii.] For all $n\in \mathbb{N}$, $S_n(x)$ has the same parity as $n$,
that is, 
\begin{equation}
S_n(-x) = (-1)^nS_n(x).  \label{eq:symmetry}
\end{equation}

\item[iii.] $(S_n)_{n\geq 0}$ satisfies a TTRR 
\begin{equation}
xS_n(x) = S_{n+1}(x)+\gamma_n S_{n-1}(x), \quad n \geq 0,  \label{eq:3trrSn}
\end{equation}
with initial conditions $S_{0}(x)=1$ and $S_{1}(x)=x$.
\end{enumerate}
\end{theorem}

An iteration of (\ref{eq:3trrFreud}) yields%
\begin{equation}
x^{2}P_{n}(x;z)=P_{n+2}(x;z)+(\gamma _{n+1}(z)+\gamma
_{n}(z))P_{n}(x;z)+\gamma _{n}(z)\gamma _{n-1}(z)P_{n-2}(x;z)
\label{eq:3trrSn2}
\end{equation}%
as well as%
\begin{align}
x^{3}P_{n}(x;z)& =P_{n+3}(x;z)+(\gamma _{n}(z)+\gamma _{n+1}(z)+\gamma
_{n+2}(z))P_{n+1}(x;z)  \notag \\
& \quad +\gamma _{n}(z)(\gamma _{n}(z)+\gamma _{n+1}(z)+\gamma
_{n-1}(z))P_{n-1}(x;z)  \label{eq:3trrSn3} \\
& \quad +(\gamma _{n-2}(z)\gamma _{n-1}(z)\gamma _{n}(z))P_{n-3}(x;z). 
\notag
\end{align}%
If we write%
\begin{equation}
P_{n}(x;z)=x^{n}+\lambda _{n,n-2}(z)x^{n-2}+\lambda _{n,n-4}(z)x^{n-4}+%
\mathcal{O}(x^{n-6}),  \label{eq:Sn}
\end{equation}%
with $\lambda _{n,k}(z)=0$ for all $k<n$, then from (\ref{eq:3trrSn}) and (%
\ref{eq:Sn}) we get%
\begin{equation*}
xP_{n}(x;z)=P_{n+1}(x;z)+(\lambda _{n+1,n-1}(z)+\gamma _{n}(z))x^{n-1}+%
\mathcal{O}(x^{n-3}),
\end{equation*}%
while from (\ref{eq:3trrSn2}) and (\ref{eq:Sn}) we obtain%
\begin{equation*}
x^{2}P_{n}(x;z)=x^{n+2}+[\lambda _{n+2,n}(z)+\gamma _{n}(z)+\gamma
_{n+1}(z)]x^{n}+\mathcal{O}(x^{n-2}).
\end{equation*}%
By multiplying (\ref{eq:Sn}) by $x$ and $x^{2}$, and comparing coefficients
of $x$, we have%
\begin{equation}
\left\{ 
\begin{array}{ll}
\gamma _{n}(z)=\lambda _{n,n-2}(z)-\lambda _{n+1,n-1}(z), &  \\ 
\gamma _{n}(z)+\gamma _{n+1}(z)=\lambda _{n,n-2}(z)-\lambda _{n+2,n}(z). & 
\end{array}%
\right.  \label{eq:sistgamma}
\end{equation}%
Moreover, using (\ref{eq:sistgamma}) 
\begin{align}
\partial _{x}P_{n}(x;z)& =nP_{n-1}(x;z)+\left( (n-2)\lambda
_{n,n-2}(z)-n\lambda _{n-1,n-3}(z)\right) P_{n-3}(x;z)+\mathcal{O}(x^{n-5}) 
\notag \\
& =nP_{n-1}(x;z)-\left( n\gamma _{n-1}(z)+2\lambda _{n,n-2}(z)\right)
P_{n-3}(x;z)+\mathcal{O}(x^{n-5}).  \label{eq:derivP}
\end{align}

\section{Truncated symmetric Freud linear functional}

In this section we study the truncated Freud linear functional $\bm{u}$
defined by (\ref{eq:FreudLF}), which is a $D-$semiclassical functional of
class $\bm{\mathfrak{s}}(\bm{u}) = 4$. In order to prove it, we will first
find the corresponding Pearson equation. Next, we will deduce a fourth-order
linear recurrence equation satisfied by its moments. As a straightforward
consequence we obtain a first-order non-homogeneous linear differential
equation that the Stieltjes function associated with $\bm{u}$ satisfies.

\subsection{Pearson equation}

\begin{proposition}
$\bm{u}$ satisfies the Pearson equation 
\begin{equation}
D(\phi(x;z)\bm{u})+\psi(x;z)\bm{u}=0 ,  \label{eq:PearsonFreud}
\end{equation}
with $\phi(x;z) = x^2-z^2$ and $\psi(x;z) = 4x^3(x^2-z^2)-2x$. As a
consequence, $\bm{u}$ is $D-$semiclassical of class four.
\end{proposition}

\begin{proof}
Using integration by parts 
\begin{align*}
\bigl< D(\phi(x;z) \bm{u}),p(x)\bigr> =& -\bigl<\bm{u}, \phi(x;z)p^{\prime
}(x)\bigr> = - \displaystyle\int^z_{-z}\phi(x;z)p^{\prime -x^4}dx \\
=& \left.-\phi(x;z)p(x)e^{-x^4}\right|^z_{-z}+\displaystyle%
\int^z_{-z}p(x)\left(\phi^{\prime 3}\phi(x;z)\right)e^{-x^4}dx \\
=&\displaystyle\int_{-z}^zp(x)\left(2x-4x^3(x^2-z^2)\right)e^{-x^4}dx \\
=&-\bigl< \bm{u},\psi(x;z)p(x)\bigr>.
\end{align*}
\end{proof}

% \begin{remark}
%     For the linear functional related with the Freud polynomials, we have the Pearson equation
%     \begin{equation}
%         D(\bm{u}_F)+4x^3\bm{u}_F=0
%         \label{eq:PearsonFreud}
%     \end{equation}
%     multiplying (\ref{eq:PearsonFreud}) by $\phi(x,z)$, we get
% \end{remark}

\subsection{Moments}

From the definition of the linear functional $\bm{u}$, the moments $(\bm{u}%
_{n})_{n\geq 0}$ are given by 
\begin{equation*}
\bm{u}_{2n+1}=\displaystyle\int_{-z}^{z}x^{2n+1}e^{-x^{4}}dx=0,\quad n\geq 0,
\end{equation*}%
and setting $s=x^{4}$, from (\ref{eq:FreudLF}) one gets 
\begin{equation}
\bm{u}_{2n}=2\displaystyle\int_{0}^{z}x^{2n}e^{-x^{4}}dx=\frac{1}{2}%
\displaystyle\int_{0}^{z^{4}}s^{\frac{2n-3}{4}}e^{-s}ds=\frac{1}{2}\hat{%
\gamma}(\frac{2n+1}{4},z^{4}),\quad n\geq 0.  \label{eq:evenmomentsTF}
\end{equation}%
In general 
\begin{equation}
\bm{u}_{n}=\frac{\Delta _{n}}{4}\hat{\gamma}\bigl(\frac{n+1}{4},z^{4}\bigr)%
,\quad \Delta _{k}=1+(-1)^{k},  \label{eq:un}
\end{equation}%
where $\hat{\gamma}(a,z)$ is the lower incomplete gamma function defined by 
\begin{equation*}
\hat{\gamma}(a,z)=\displaystyle\int_{0}^{z}x^{a-1}e^{-x}dx.
\end{equation*}%
Its series representation is 
\begin{equation}
\hat{\gamma}(a,z)=\frac{z^{a}e^{-z}}{a}\displaystyle\sum_{k=0}^{\infty }%
\frac{z^{k}}{(a+1)_{k}},  \label{eq:inclowergammaS}
\end{equation}%
where $(c)_{k}$ is the Pochhammer symbol defined by 
\begin{equation*}
(c)_{k}=\displaystyle\prod_{j=0}^{k-1}(c+j),\quad k\geq 1,\quad (c)_{0}=1.
\end{equation*}%
\smallskip In particular, replacing (\ref{eq:inclowergammaS}) in (\ref%
{eq:evenmomentsTF}) we get%
\begin{equation*}
\bm{u}_{2n}=\frac{2}{{2n+1}}z^{2n+1}e^{-z^{4}}\sum_{k=0}^{\infty }\frac{%
z^{4k}}{(\frac{2n+5}{4})_{k}},
\end{equation*}%
and from the recurrence 
\begin{equation*}
\hat{\gamma}(a+1,z)=a\hat{\gamma}(a,z)-z^{a}e^{-z},
\end{equation*}%
we obtain 
\begin{equation}
4\bm{u}_{2n+4}=(2n+1)\bm{u}_{2n}-2z^{2n+1}e^{-z^{4}}.
\label{eq:recumomentsexp}
\end{equation}

From Pearson's equation (\ref{eq:PearsonFreud}), a fourth-term recurrence
relation for the moments can be deduced. Indeed,

\begin{proposition}
Let $(\bm{u}_{2n})_{n\geq 0}$ be the sequence of even moments of $\bm{u}$.
This sequence satisfies the recurrence 
\begin{equation}
4 \bm{u}_{2n+6} -4z^2\bm{u}_{2n+4}-(2n+3)\bm{u}_{2n+2}+(2n+1)z^2\bm{u}%
_{2n}=0,  \label{eq:recumoment}
\end{equation}
with initial conditions 
\begin{equation*}
\bm{u}_0 = \frac{1}{2}\hat{\gamma}(\frac{1}{4},z^4), \quad \bm{u}_2 = \frac{1%
}{2}\hat{\gamma}(\frac{3}{4},z^4),\quad \bm{u}_4 = \frac{\bm{u}_0}{4}-\frac{1%
}{2}ze^{z^4}.
\end{equation*}
\end{proposition}

\begin{proof}
From (\ref{eq:PearsonFreud}), we have%
\begin{align*}
\bigl<D(\phi (x;z)\bm{u})+\psi (x;z)\bm{u},x^{2n+1}\bigr>& =-\bigl<\bm{u}%
,(2n+1)\phi (x;z)x^{2n}\bigr>+\bigl<\bm{u},\psi (x;z)x^{2n+1}\bigr> \\
& =-\bigl<\bm{u},(2n+1)(x^{2}-z^{2})x^{2n}\bigr>+\bigl<\bm{u}%
,(4x^{3}(x^{2}-z^{2})-2x)x^{2n+1}\bigr> \\
& =(2n+1)\bm{u}_{2n+2}+(2n+1)z^{2}\bm{u}_{2n}+4\bm{u}_{2n+6}-4z^{2}\bm{u}%
_{2n+4}-2\bm{u}_{2n+2} \\
& =4\bm{u}_{2n+6}-4z^{2}\bm{u}_{2n+4}-(2n+3)\bm{u}_{2n+2}+(2n+1)z^{2}\bm{u}%
_{2n}=0.
\end{align*}
\end{proof}

Finally, from the asymptotic expansion of the lower incomplete gamma function%
\begin{equation*}
\hat{\gamma}(a,z)\sim \Gamma (a)\left( 1-z^{a-1}e^{-z}\sum_{k=0}^{\infty }%
\frac{z^{-k}}{\Gamma (a-k)}\right) ,\quad z\rightarrow \infty ,
\end{equation*}%
and from the expression of the even moments of the Freud polynomials 
\begin{equation*}
\bm{u}_{2n}^{F}=\frac{1}{2}\displaystyle\int_{0}^{\infty }s^{\frac{2n-3}{4}%
e^{-s}ds}=\frac{1}{2}\Gamma (\frac{2n+1}{4}),
\end{equation*}%
we get the following asymptotic expansion for the ratio of moments 
\begin{equation*}
\frac{\bm{u}_{2n}}{\bm{u}_{2n}^{F}}\sim 1-e^{-z^{4}}\displaystyle%
\sum_{k=0}^{\infty }\frac{z^{2(n-2k)-3}}{\Gamma (\frac{2n+1}{4}-k)},\quad
z\rightarrow \infty .
\end{equation*}

\subsection{Stieltjes function}

\begin{definition}
The Stieltjes function associated with a linear functional $\bm{\ell}$ with
moments $(\bm{\ell}_{n})_{n\geq 0}$ is the formal series defined by 
\begin{equation*}
\mathbf{S}_{\bm{\ell}}(t)=\sum_{k=0}^{\infty }\frac{\bm{\ell}_{n}}{t^{n+1}}.
\end{equation*}
\end{definition}

\begin{proposition}
The Stieltjes function $\mathbf{S}_{\bm{u}}(t;z)$ associated with the linear
functional $\bm{u}$ satisfies the following first order non-homogeneous
linear ordinary differential equation 
\begin{equation}
\phi(t;z) \partial_t \mathbf{S}_{\bm{u}}(t;z) + [\phi ^{\prime
}(t;z)+\psi(t;z)]\mathbf{S}_{\bm{u}}(t;z) = \bm{u}_0(4t^4-1)+4\left(\bm{u}%
_2t^2+\bm{u}_4\right).  \label{eq:StieltjesODE}
\end{equation}
Here $\phi(t;z)$ and $\psi(t;z)$ are given in $(\ref{eq:Pearson})$.
\end{proposition}

\begin{proof}
First note that%
\begin{equation}
\mathbf{S}_{\bm{u}}(t;z)=\sum_{n=0}^{\infty }\frac{\bm{u}_{2n}}{t^{2n+1}},
\label{eq:StieltjesFreud}
\end{equation}%
and, therefore,%
\begin{equation*}
\partial _{t}\mathbf{S}_{\bm{u}}(t;z)=-\sum_{n=0}^{\infty }(2n+1)\frac{\bm{u}%
_{2n}}{t^{2n+2}}.
\end{equation*}%
Then from (\ref{eq:recumoment})%
\begin{align*}
\mathbf{S}_{\bm{u}}(t;z)& =\frac{\bm{u}_{0}}{t}+\frac{\bm{u}_{2}}{t^{3}}+%
\frac{\bm{u}_{4}}{t^{5}}+\displaystyle\sum_{n=3}^{\infty }\frac{\bm{u}_{2n}}{%
t^{2n+1}} \\
& =\frac{\bm{u}_{0}}{t}+\frac{\bm{u}_{2}}{t^{3}}+\frac{\bm{u}_{4}}{t^{5}}+%
\displaystyle\sum_{n=1}^{\infty }\frac{\bm{u}_{2n+4}}{t^{2n+5}} \\
& =\frac{\bm{u}_{0}}{t}+\frac{\bm{u}_{2}}{t^{3}}+\frac{\bm{u}_{4}}{t^{5}}%
+z^{2}\displaystyle\sum_{n=1}^{\infty }\frac{\bm{u}_{2n+2}}{t^{2n+5}}+\frac{1%
}{4}\displaystyle\sum_{n=1}^{\infty }(2n+1)\frac{\bm{u}_{2n}}{t^{2n+5}}-%
\frac{z^{2}}{4}\displaystyle\sum_{n=1}^{\infty }(2n-1)\frac{\bm{u}_{2n-2}}{%
t^{2n+5}} \\
& =\frac{\bm{u}_{0}}{t}+\frac{\bm{u}_{2}}{t^{3}}+\frac{\bm{u}_{4}}{t^{5}}+%
\frac{z^{2}}{t^{2}}\displaystyle\sum_{n=0}^{\infty }\frac{\bm{u}_{2n}}{%
t^{2n+1}}+\frac{1}{4t^{3}}\left( \displaystyle\sum_{n=1}^{\infty }(2n+1)%
\frac{\bm{u}_{2n}}{t^{2n+2}}+\frac{\bm{u}_{0}}{t^{2}}-\frac{\bm{u}_{0}}{t^{2}%
}\right) -\frac{z^{2}}{4t^{5}}\displaystyle\sum_{n=0}^{\infty }(2n+1)\frac{%
\bm{u}_{2n}}{t^{2n+2}} \\
& =\frac{\bm{u}_{0}}{t}+\frac{\bm{u}_{2}}{t^{3}}+\frac{\bm{u}_{4}}{t^{5}}+%
\frac{z^{2}}{t^{2}}\mathbf{S}_{\bm{u}}(t;z)-\frac{1}{4t^{3}}\partial _{t}%
\mathbf{S}_{\bm{u}}(t;z)-\frac{\bm{u}_{0}}{4t^{5}}+\frac{z^{2}}{4t^{5}}%
\partial _{t}\mathbf{S}_{\bm{u}}(t;z).
\end{align*}%
Thus, multiplying both sides by $-4t^{5}$ we get 
\begin{equation*}
(t^{2}-z^{2})\partial _{t}\mathbf{S}_{\bm{u}}(t;z)+4t^{3}(t^{2}-z^{2})%
\mathbf{S}_{\bm{u}}(t;z)=\bm{u}_{0}(4t^{4}-1)+4\left( \bm{u}_{2}t^{2}+\bm{u}%
_{4}\right) .
\end{equation*}
\end{proof}

% \begin{remark}
%     Note that differentiating (\ref{eq:StieltjesODE}) three times with respect to $t$ we have
%     \begin{equation*}
%         \partial^3_t\left[(t^2-z^2)\left(\partial_t\mathbf{S}_{\bm{u}}(t;z)+4t^3\mathbf{S}_{\bm{u}}(t;z)\right)\right] = 96tu_0(z)
%     \end{equation*}
%     and therefore,
%     \begin{equation*}
%         \partial_t\left(\frac{\partial^3_t\left[(t^2-z^2)\left(\partial_t\mathbf{S}_{\bm{u}}(t;z)+4t^3\mathbf{S}_{\bm{u}}(t;z)\right)\right]}{t}\right) = 0
%     \end{equation*}
%     Thus, the Stieltjes function (\ref{eq:StieltjesODE}) satisfies the fifth order homogeneous linear ODE with polynomial coefficients
%     \begin{align*}
%         &t(t^2-z^2)\partial^5_t\mathbf{S}_{\bm{u}}(t;z)+[4t^6-4t^4z^2+7t^2+z^2]\partial^4_t\mathbf{S}_{\bm{u}}(t;z)+[76t^5-44t^3z^2+6t]\partial^3_t\mathbf{S}_{\bm{u}}(t;z) \\
%         &+[420t^4-108t^2z^2-6]\partial^2_t\mathbf{S}_{\bm{u}}(t;z)+24t(30t^2-z^2)\partial_t\mathbf{S}_{\bm{u}}(t;z)+24(10t^2+z^2)\mathbf{S}_{\bm{u}}(t;z)=0
%     \end{align*}
% \end{remark}

\section{Laguerre-Freud equations}

We start this section by obtaining a third-order nonlinear difference
equation that the coefficients of the TTRR satisfy.

\begin{theorem}
The coefficients $\gamma_n(z)$ in (\ref{eq:3trrFreud}) associated with the
linear functional $\bm{u}$ satisfy the Laguerre-Freud equation 
\begin{align}
\frac{z^2}{4} &= \gamma_n\left[\gamma_{n-1}\left(\gamma_{n-2}+\gamma_{n-1}+%
\gamma_n-z^2\right)+\gamma_n\left(\gamma_{n+1}+\gamma_n+\gamma_{n-1}-z^2%
\right)+\frac{1}{4}-\frac{n}{2}\right]  \notag \\
&\quad -\gamma_{n+1}\left[\gamma_{n+2}\left(\gamma_{n+3}+\gamma_{n+2}+%
\gamma_{n+1}-z^2\right)+\gamma_{n+1}\left(\gamma_{n+2}+\gamma_{n+1}+%
\gamma_n-z^2\right)-\frac{n}{2}-\frac{3}{4}\right].  \label{eq:FreudLaguerre}
\end{align}
\end{theorem}

\begin{proof}
Let $h_n = \bigl< \bm{u},P_n^2\bigr>$. Since $\bm{u}$ satisfies (\ref%
{eq:Pearson}) with $\phi(x;z) = x^2-z^2$ and \newline
$\psi(x;z) = 4x^3(x^2-z^2)-2x$, then 
\begin{equation}
\bigl< \bm{u}, \psi P_nP_{n+1}\bigr> = -\bigl< D(\phi\bm{u}),P_nP_{n+1}%
\bigr> = \bigl< \bm{u},\phi (\partial_x P_n)P_{n+1}\bigr> + \bigl< \bm{u}%
,\phi P_n\partial_x P_{n+1}\bigr>.  \label{eq:FLeq}
\end{equation}

Using (\ref{eq:3trrSn2}), (\ref{eq:derivP}) and taking into account (\ref%
{eq:gammah_m}) we have

\begin{align*}
\bigl<\bm{u},\phi (\partial_x P_n)P_{n+1}\bigr> = \bigl<\bm{u}%
,x^2(\partial_x P_n)P_{n+1}\bigr> =n\bigl< \bm{u},P^2_{n+1}\bigr> =
n\gamma_{n+1}h_n,
\end{align*}

\begin{align*}
\bigl< \bm{u}, \phi P_n\partial_x P_{n+1}\bigr> &= \bigl< \bm{u}, x^2
P_n\partial_x P_{n+1}\bigr> - z^2\bigl< \bm{u}, P_n\partial_x P_{n+1}\bigr>
\\
&= \bigl< \bm{u}, x^2P_n\left[(n+1)P_n-((n+1)\gamma_n+2%
\lambda_{n+1,n-1})P_{n-2}\right]\bigr> -z^2(n+1)h_n \\
&= (\gamma_n+\gamma_{n+1})(n+1)\bigl<\bm{u},P^2_n\bigr>-\gamma_n%
\gamma_{n-1}((n+1)\gamma_n+2\lambda_{n+1,n-1})\bigl< \bm{u}, P^2_{n-2}\bigr>%
-z^2(n+1)h_n \\
&= [(n+1)(\gamma_{n+1}-z^2)-2\lambda_{n+1,n-1}]h_n.
\end{align*}

\bigskip Using (\ref{eq:3trrSn2}), (\ref{eq:3trrSn3}) and (\ref{eq:gammah_m}%
) 
\begin{align*}
\bigl< \bm{u},\psi P_nP_{n+1} \bigr> &= 4\bigl< \bm{u},(x^3P_n)(x^2P_{n+1})%
\bigr> - 4z^2\bigl< \bm{u},x^3P_nP_{n+1}\bigr> -2\bigl< \bm{u},xP_nP_{n+1}%
\bigr> \\
& = 4\bigl<\bm{u},
x^3P_n(P_{n+3}+(\gamma_{n+1}+\gamma_{n+2})P_{n+1}+\gamma_{n}%
\gamma_{n+1}P_{n-1})\bigr>-4z^2\bigl<\bm{u},(\gamma_n+\gamma_{n+1}+%
\gamma_{n+2})P^2_{n+1}\bigr> \\
&-2\bigl<\bm{u},P^2_{n+1}\bigr> \\
\\
&=4\left[\gamma_{n+3}\gamma_{n+2}\gamma_{n+1}+(\gamma_{n+2}+\gamma_{n+1})(%
\gamma_n+\gamma_{n+1}+\gamma_{n+2})\gamma_{n+1}+\gamma_n\gamma_{n+1}(%
\gamma_n+\gamma_{n+1}+\gamma_{n-1})\right]h_n \\
&- 4z^2\gamma_{n+1}(\gamma_n+\gamma_{n+1}+\gamma_{n+2})h_n -
2\gamma_{n+1}h_n.
\end{align*}
Then from (\ref{eq:FLeq}) 
\begin{align}
&4\left[\gamma_{n+3}\gamma_{n+2}\gamma_{n+1}+(\gamma_{n+2}+\gamma_{n+1})(%
\gamma_n+\gamma_{n+1}+\gamma_{n+2})\gamma_{n+1}+\gamma_n\gamma_{n+1}(%
\gamma_n+\gamma_{n+1}+\gamma_{n-1})\right]  \notag \\
&- 4z^2\gamma_{n+1}(\gamma_n+\gamma_{n+1}+\gamma_{n+2}) =
(2n+3)\gamma_{n+1}-z^2(n+1)-2\lambda_{n+1,n-1}.  \label{eq:Freudn}
\end{align}
Shifting $n\rightarrow n-1$

\begin{align}
&4\left[\gamma_{n+2}\gamma_{n+1}\gamma_{n}+(\gamma_{n+1}+\gamma_{n})(%
\gamma_{n-1}+\gamma_{n}+\gamma_{n+1})\gamma_{n}+\gamma_{n-1}\gamma_{n}(%
\gamma_{n-1}+\gamma_{n}+\gamma_{n-2})\right]  \notag \\
&- 4z^2\gamma_{n}(\gamma_{n-1}+\gamma_{n}+\gamma_{n+1}) =
(2n+1)\gamma_{n}-nz^2-2\lambda_{n,n-2}.  \label{eq:Freudn-1}
\end{align}

Subtracting (\ref{eq:Freudn}) - (\ref{eq:Freudn-1}) and using (\ref%
{eq:sistgamma}) we finally get (\ref{eq:FreudLaguerre}).
\end{proof}

\medskip

Notice that (\ref{eq:FreudLaguerre}) is a third-order difference equation
involving six consecutive terms of the sequence $(\gamma _{n}(z))_{n\geq 0}$%
. To achieve a more accurate calculation of $\gamma _{n}(z)$, we now derive
a fifth-order difference relation in terms of five consecutive terms.

\begin{theorem}
The coefficients $\gamma _{n}(z)$ satisfy the nonlinear difference relation%
\begin{align}
& z^{2}\left( \frac{n}{2}-2\gamma _{n}(\gamma _{n+1}+\gamma _{n}+\gamma
_{n-1})\right) ^{2}=  \notag \\
& \gamma _{n}\left( n+\frac{1}{2}-2\gamma _{n}(\gamma _{n+1}+\gamma
_{n}+\gamma _{n-1})-2\gamma _{n+1}(\gamma _{n+2}+\gamma _{n+1}+\gamma
_{n})\right)   \label{eq:Freud2} \\
& \left( n-\frac{1}{2}-2\gamma _{n}(\gamma _{n+1}+\gamma _{n}+\gamma
_{n-1})-2\gamma _{n-1}(\gamma _{n}+\gamma _{n-1}+\gamma _{n-2})\right)
,\quad n\geq 0.  \notag
\end{align}%
Here we assume $\gamma _{0}=\gamma _{-1}=\gamma _{-2}=0.$
\end{theorem}

%%%%%%%%%%%%%%%%%%%%%%%%%

%%%%%%%%%%%%%%%%%%%%%%%%%

\begin{proof}
From (\ref{eq:3trrSn3}) we have%
\begin{equation*}
\bigl<\bm{u},x^{3}P_{n}P_{n-1}\bigr>=\gamma _{n}(\gamma _{n+1}+\gamma
_{n}+\gamma _{n-1})h_{n-1},\quad n\geq 1.
\end{equation*}%
Integrating by parts we get%
\begin{equation*}
\bigl<\bm{u},x^{3}P_{n}P_{n-1}\bigr>=-\frac{1}{4}\left[
P_{n}P_{n-1}e^{-x^{4}}\right] _{-z}^{z}+\frac{1}{4}\bigl<\bm{u},\partial
_{x}(P_{n}P_{n-1})\bigr>,
\end{equation*}%
while (\ref{eq:symmetry}) gives%
\begin{equation*}
\left[ P_{n}P_{n-1}e^{-x^{4}}\right]
_{-z}^{z}=2P_{n}(z;z)P_{n-1}(z;z)e^{-z^{4}}.
\end{equation*}%
Using (\ref{eq:OPS}) we see that $\bigl<\bm{u},P_{n}\partial _{x}P_{n-1}%
\bigr>=0$ and since $P_{n}(x;z)=x^{n}+\mathcal{O}(x^{n-1})$ we have%
\begin{equation*}
\bigl<\bm{u},P_{n-1}\partial _{x}P_{n}\bigr>=nh_{n-1}(z).
\end{equation*}%
Hence,%
\begin{equation}
P_{n}(z;z)P_{n-1}(z;z)e^{-z^{4}}=\left[ \frac{n}{2}-2\gamma _{n}(\gamma
_{n+1}+\gamma _{n}+\gamma _{n-1})\right] h_{n-1}.  \label{eq:Pn(z;z)}
\end{equation}%
On the other hand, from (\ref{eq:3trrSn2}) and (\ref{eq:OPS}) we get 
\begin{equation*}
\bigl<\bm{u},x^{4}P_{n}^{2}\bigr>=\bigl<\bm{u},(x^{2}P_{n})^{2}\bigr>%
=h_{n+2}+(\gamma _{n}+\gamma _{n+1})^{2}h_{n}+\gamma _{n}^{2}\gamma
_{n-1}^{2}h_{n-2}=\bigl[\gamma _{n+2}\gamma _{n+1}+(\gamma _{n}+\gamma
_{n+1})^{2}+\gamma _{n}\gamma _{n-1}\bigr]h_{n}.
\end{equation*}%
An integration by parts gives 
\begin{equation}
\bigl<\bm{u},\partial _{x}(xP_{n}^{2})\bigr>=\left[ xP_{n}^{2}e^{-x^{4}}%
\right] _{-z}^{z}+4\bigl<\bm{u},x^{4}P_{n}^{2}\bigr>%
=2zP_{n}^{2}(z;z)e^{-z^{4}}+4\bigl<\bm{u},x^{4}P_{n}^{2}\bigr>.
\label{eq:Freud2..xP2}
\end{equation}%
As a consequence, 
\begin{equation}
\bigl<\bm{u},\partial _{x}(xP_{n}^{2})\bigr>=2zP_{n}^{2}(z;z)e^{-z^{4}}+4%
\bigl[\gamma _{n+2}\gamma _{n+1}+(\gamma _{n}+\gamma _{n+1})^{2}+\gamma
_{n}\gamma _{n-1}\bigr]h_{n}.  \label{eq:Freud2.xP2}
\end{equation}%
On the other hand, since 
\begin{equation*}
xP_{n}(x;z)\partial _{x}P_{n}(x;z)=P_{n}\left[ nx^{n}+\mathcal{O}(x^{n-1})%
\right] ,
\end{equation*}%
and using (\ref{eq:OPS}), we obtain 
\begin{equation*}
\bigl<\bm{u},\partial _{x}(xP_{n}^{2})\bigr>=\bigl<\bm{u},P_{n}^{2}\bigr>+2%
\bigl<\bm{u},xP_{n}\partial _{x}P_{n}\bigr>=(2n+1)h_{n}.
\end{equation*}%
From (\ref{eq:Freud2.xP2}) and (\ref{eq:Freud2..xP2}) we conclude that
\begin{equation}
P^2_n(z;z)e^{-z^4} = \frac{(2n+1)-4\bigl[\gamma_{n+2}\gamma_{n+1}+(\gamma_n+%
\gamma_{n+1})^2 +\gamma_{n}\gamma_{n-1}\bigr] }{2z}h_n.  \label{eq:Pn(z;z)^2}
\end{equation}
Squaring (\ref{eq:Pn(z;z)}), we have 
\begin{equation*}
P^2_n(z;z)e^{-z^4}P^2_{n-1}(z;z)e^{-z^4} = \left[\frac{n}{2}%
-2\gamma_n(\gamma_{n+1}+\gamma_n+\gamma_{n-1})\right]^2 h^2_{n-1},
\end{equation*}
and using (\ref{eq:Pn(z;z)^2}), we obtain
\begin{align*}
&\frac{(2n+1)-4\bigl[\gamma_{n+2}\gamma_{n+1}+(\gamma_n+\gamma_{n+1})^2
+\gamma_{n}\gamma_{n-1}\bigr] }{2z}h_n \\
&\times \frac{(2n-1)-4\bigl[\gamma_{n+1}\gamma_{n}+(\gamma_{n-1}+%
\gamma_{n})^2 +\gamma_{n-1}\gamma_{n-2}\bigr] }{2z}h_{n-1} \\
&= \left[\frac{n}{2}-2\gamma_n(\gamma_{n+1}+\gamma_n+\gamma_{n-1})\right]^2
h^2_{n-1}, \quad n\geq 2.
\end{align*}
Finally, from (\ref{eq:gammah_m}),
\begin{align*}
&\frac{\gamma_n\biggl(2n-4\gamma_{n+2}\gamma_{n+1}-4(\gamma_n+%
\gamma_{n+1})^2-4\gamma_{n}\gamma_{n-1}+1\biggr)\biggl(2n-4\gamma_{n-1}%
\gamma_{n-2}-4(\gamma_{n}+\gamma_{n-1})^2-4\gamma_{n+1}\gamma_{n}-1\biggr)}{%
4z^2} \\
&= \biggl(\frac{n}{2}-2\gamma_n(\gamma_{n+1}+\gamma_n+\gamma_{n-1})\biggr)^2.
\end{align*}
That yields the result for all $n\geq 2$. Notice that \eqref{eq:Pn(z;z)} and \eqref{eq:Pn(z;z)^2}, together with the initial condition $\gamma_0=0$ give,
for $n=1$, 
\begin{equation*}
\frac{\gamma_1\biggl(3-4\gamma_3\gamma_2-4(\gamma_1+\gamma_2)^2\biggr)\biggl(%
1-4\gamma_1^2+\gamma_2\gamma_1\biggr)}{4z^2} = \biggl(\frac{1}{2}%
-2\gamma_1(\gamma_2+\gamma_1)\biggr)^2.
\end{equation*}
%Notice that, for $n=1$, the previous relation reduces to
% \begin{equation*}
%     \frac{\gamma_1\biggl(2-4\gamma_3\gamma_2-4(\gamma_1+\gamma_2)^2+1\biggr)\biggl(2-4\gamma_1^2+\gamma_2\gamma_1-1\biggr)}{4z^2} = \biggl(\frac{1}{2}-2\gamma_1(\gamma_2+\gamma_1)\biggr)^2
% \end{equation*}
\end{proof}

\begin{corollary}
For $n\in \mathbb{N}$, let $g_n(z)$ be defined by 
\begin{equation}
g_n(z) = \frac{n}{2}-2\gamma_n(z)\left(\gamma_{n+1}(z)+\gamma_n(z)+%
\gamma_{n-1}(z)\right).  \label{eq:gn}
\end{equation}
Then, $g_n(z)$ satisfies the nonlinear relation 
\begin{equation}
z^2g^2_n = \gamma_n(g_{n+1}+g_n)(g_n+g_{n-1}).
\label{eq:z^2gn^2}
\end{equation}
% Moreover, notice that (\ref{eq:gn}) and (\ref{eq:z^2gn^2}) generate the corresponding succession

% \begin{equation*}
%     g_0 = 0, \qquad
% \end{equation*}
\end{corollary}

\begin{remark}
Notice that \eqref{eq:gn} and \eqref{eq:z^2gn^2} constitute a system of
coupled difference equations that can be recursively obtained from an
iterative process with initial conditions $\gamma_0=g_0= 0$ as well as (\ref%
{eq:evenmomentsTF}),

\begin{equation*}
\gamma_1(z) = \frac{\bm{u}_2(z)}{\bm{u}_0(z)}, \quad \gamma_{2}(z) = \frac{%
\bm{u}_4(z) \bm{u}_{0} - \bm{u}^2_2(z)}{\bm{u}_{0}\bm{u}_2(z)}.
\end{equation*}
Then, from \eqref{eq:gn} we have 
\begin{equation*}
g_1(z) = \frac{1}{2}-2\gamma_1(z)(\gamma_2(z)+\gamma_1(z)).
\end{equation*}
\end{remark}

\begin{remark}
According to (\ref{eq:gammafreud}), we see that $g_n(z)\rightarrow 0$ when $%
z\rightarrow\infty.$
\end{remark}

For $k=1,2,...,n-1,$ we follow the scheme in Figure \ref{figEsquema}.

\begin{figure}[tbp]
\centering
\fbox{ 
\begin{tikzpicture}[
    node distance=0.1cm and 2cm,
    every node/.style={minimum size=1cm},
    arrow/.style={-{Stealth[length=2mm, width=2mm]}}
  ]
  \node (A1) at (-6, 4) {$g_1$};
  \node (A2) [below=of A1] {$\gamma_3$};
  \node (A3) [below=of A2] {$\gamma_4$};
  \node (A4) [below=of A3] {$\gamma_5$};
  \node (A5) [below=of A4] {$\vdots$};
  \node (A6) [below=of A5] {$\vdots$};
  \node (A7) [below=of A6] {$\gamma_{n-2}$};
  \node (A8) [below=of A7] {$\gamma_{n-1}$};
  \node (A9) [below=of A8] {$\gamma_{n}$};

    \node (B1) at (4, 4) {$g_2$};
  \node (B2) [below=of B1] {$g_3$};
  \node (B3) [below=of B2] {$g_4$};
  \node (B4) [below=of B3] {$g_5$};
  \node (B5) [below=of B4] {$\vdots$};
  \node (B6) [below=of B5] {$\vdots$};
  \node (B7) [below=of B6] {$g_{n-2}$};
  \node (B8) [below=of B7] {$g_{n-1}$};
  \node (B9) [below=of B8] {$g_{n}$};

    \draw[arrow,black] (A1) -- (B1);
  \draw [arrow, black] (A1) -- (B2);
  \draw[arrow,black] (A2) -- (B3);
  \draw[arrow,black] (A3) -- (B4);
  \draw [arrow, black] (B1) -- (A2);
  \draw[arrow,black] (B2) -- (A3);
  \draw[arrow,black] (B3) -- (A4);
  \draw[arrow, black] (B1) -- (B2);
  \draw[arrow,black] (B2) -- (B3);
  \draw[arrow,black] (B3) -- (B4);
  \draw[arrow,black] (A1) -- (A2);
  \draw[arrow,black] (A2) -- (A3);
  \draw[arrow,black] (A3) -- (A4);
  \draw[arrow,black] (A7) -- (A8);
  \draw[arrow,black] (A8) -- (A9);
  \draw[arrow,black] (B7) -- (B8);
  \draw[arrow,black] (B8) -- (B9);
  \draw[arrow,black] (B7)--(A8);
  \draw[arrow,black] (B8)--(A9);
  \draw[arrow,black] (A7)--(B8);
  \draw[arrow,black] (A8)--(B9);
  \draw[arrow,black] (A4)--(A5);
  \draw[arrow,black] (B4)--(B5);
  \draw[arrow,black] (A5)--(A6);
  \draw[arrow,black] (B5)--(B6);
  \draw[arrow,black] (A6)--(A7);
  \draw[arrow,black] (B6)--(B7);
  \draw[arrow,black] (A4)--(B5);
  \draw[arrow,black] (B4)--(A5);
  \draw[arrow,black] (A6)--(B7);
  \draw[arrow,black] (B6)--(A7);

  \draw [arrow, bend left=60,black] (B1) to (B3);
  \draw [arrow, bend left=60,black] (B2) to (B4);
  \draw [arrow, bend right=60,black] (A2) to (A4);
  \draw [arrow, bend left=60,black] (B7) to (B9);
  \draw [arrow, bend right=60,black] (A7) to (A9);
  \draw [arrow, bend left=60,black] (B6) to (B8);
  \draw [arrow, bend right=60,black] (A6) to (A8);
\draw [arrow, bend left=60,black] (B3) to (B5);
  \draw [arrow, bend right=60,black] (A3) to (A5);
    \node[anchor=east] at (-8,4) {$k = 1$};
  \node[anchor=east] at (-8,3) {$k = 2$};
  \node[anchor=east] at (-8,2) {$k = 3$};
  \node[anchor=east] at (-8,1) {$k = 4$};
  \node[anchor=east] at (-8,-0.45) {$\vdots$};
  \node[anchor=east] at (-8,-1.6) {$\vdots$};
  \node[anchor=east] at (-7.31, -2.6) {$k=n-3$};
  \node[anchor=east] at (-7.31, -3.6) {$k=n-2$};
  \node[anchor=east] at (-7.31, -4.7) {$k=n-1$};

    \node[anchor=south] at (-5,5) {$g_k = \frac{k}{2}-2\gamma_k(\gamma_{k+1}+\gamma_k+\gamma_{k-1})$};
  \node[anchor=south] at (3,5) {$z^2g_k^2 = \gamma_k(g_k+g_{k+1})(g_k+g_{k-1})$ };
\end{tikzpicture}
}
\caption{\textit{Schematic representation of the algorithm used to generate $%
\protect\gamma_n$ and $g_n$ from the coupled difference equations 
\eqref{eq:gn} and \eqref{eq:z^2gn^2} taking into account the initial
conditions $\protect\gamma_0= g_0=0$, and $\protect\gamma_1$, and $\protect%
\gamma_2$ defined as above. On the left side, we have $g_1$ and $\protect%
\gamma_{k+1}$ for $k = 1, 2, \ldots, n-1$, generated using \eqref{eq:gn}. On
the right side, we have $g_{k+1}$ for $k = 1, 2, \ldots, n-1$, generated
using \eqref{eq:z^2gn^2}. The black arrows indicate the iterative process.}}
\label{figEsquema}
\end{figure}
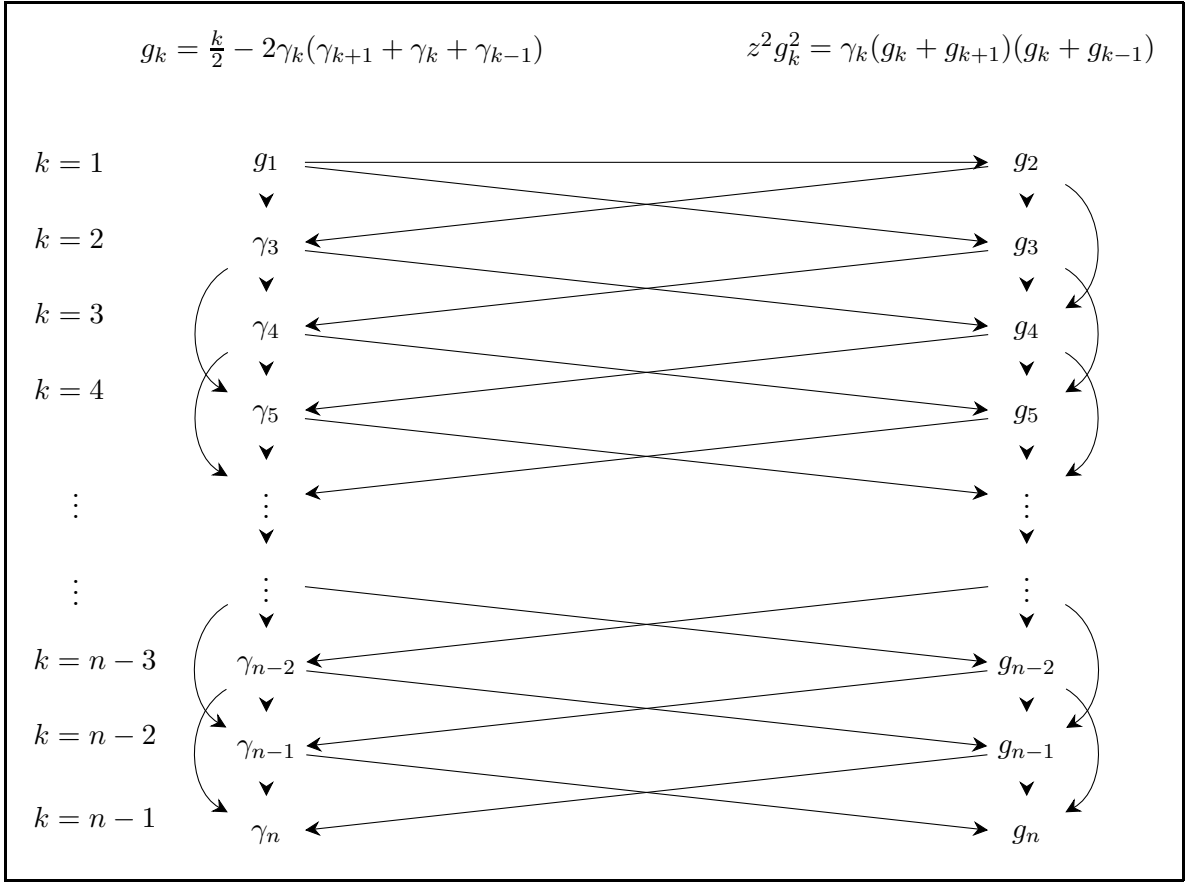
% \begin{algorithm}
% \caption{Algorithm to recursively obtain $\gamma_n(z)$ and $g_n(z)$.}\label{alg:cap}
% \begin{algorithmic}[1]
% \Require $n \geq 1$, $z\geq 0$
% \Ensure $\gamma_n(z)$, $g_n(z)$

% \State Initial Data: $\gamma_0 = g_0=0$, $\gamma_1(z)$ and $\gamma_2(z)$
% \State $X \gets x$
% \State $N \gets n$
% \While{$N \neq 0$}
% \If{$N$ is even}
%     \State $X \gets X \times X$
%     \State $N \gets \frac{N}{2}$  \Comment{This is a comment}
% \ElsIf{$N$ is odd}
%     \State $y \gets y \times X$
%     \State $N \gets N - 1$
% \EndIf
% \EndWhile
% \end{algorithmic}
% \end{algorithm}
%\end{remark}%aqui

From $(\ref{eq:z^2gn^2})$ we obtain an asymptotic expansion of $g_n(z)$ as $%
n\rightarrow\infty.$

\begin{theorem}
For $z=\mathcal{O}(1)$, we have 
\begin{equation}
g_n(z)\sim \frac{n}{2}-\frac{3z^4}{8}-\frac{3z^4}{16}n^{-2}-\frac{9z^8}{32}%
n^{-3}-\frac{3z^4(22+27z^8)}{256}n^{-4}+\mathcal{O}(n^{-5}),
\label{eq:gnasymp}
\end{equation}
as $n\rightarrow \infty$.
\end{theorem}

\begin{proof}
Replacing (\ref{eq:z^2gn^2}) in (\ref{eq:gn}), we have 
\begin{align}
&\left(\frac{n}{2}-g_n\right)\left(g_{n+2}+g_{n+1}\right)%
\left(g_{n-1}+g_{n-2}\right)\left(g_{n+1}+g_n
\right)^2\left(g_n+g_{n-1}\right)^2  \notag \\
&= 2z^4g^2_n\biggl[g^2_{n+1}\left(g_n+g_{n-1}\right)\left(g_{n-1}+g_{n-2}%
\right)  \notag \\
&\quad +g^2_n\left(g_{n+2}+g_{n+1}\right)\left(g_{n-1}+g_{n-2}\right)  \notag
\\
&\quad +g^2_{n-1}\left(g_{n+2}+g_{n+1}\right)\left(g_{n+1}+g_{n}\right)%
\biggr].  \label{eq:gngn}
\end{align}
Next, we can replace the expansion 
\begin{equation*}
g_n(z) = \displaystyle\sum_{k\geq -1}\frac{\xi_k(z)}{n^k},
\end{equation*}
in (\ref{eq:gngn}), and a comparison of coefficients of the powers in $n$, yields 
\begin{equation*}
32\xi^6_{-1}(1-2\xi_{-1})=0, \quad \mathcal{O}(n^7).
\end{equation*}
Therefore $\xi_{-1}=0$, or $\xi_{-1}=1/2$. The solution $\xi_{-1}=0$ leads
to $\xi_k(z)=0$ for all $k\geq -1$, and hence we choose $\xi_{-1}=1/2$. For $%
\mathcal{O}(n^6)$, we have 
\begin{equation*}
-\left(\xi_0+\frac{3z^4}{8}\right) =0\longrightarrow \xi_0 = -\frac{3z^4}{8}.
\end{equation*}
The next term, for $\mathcal{O} (n^5)$ gives 0, and continuing this way we
obtain (\ref{eq:gnasymp}).
\end{proof}

\begin{corollary}
For $z=\mathcal{O}(1)$, the following asymptotic expansion holds 
\begin{equation*}
2\gamma_n(z)(\gamma_{n+1}(z)+\gamma_n(z)+\gamma_{n-1}(z)) \sim \frac{3z^4}{8}%
+\frac{3z^4}{16}n^{-2}+\frac{9z^8}{32}n^{-3}+\frac{3z^4(22+27z^8)}{256}%
n^{-4}+\mathcal{O}(n^{-5}).
\end{equation*}
Replacing 
\begin{equation*}
\gamma_n(z) = \sum_{k\geq 0} \frac{\omega_k(z)}{n^k}
\end{equation*}
in (\ref{eq:Freud2}) gives

\begin{equation*}
\gamma_n(z) \sim \frac{z^2}{4}+\frac{z^2}{16}n^{-2}+\frac{3z^6}{32}n^{-3}+%
\frac{z^2\left(4-8z^4+27z^6\right)}{256}n^{-4}+\mathcal{O}(n^{-5}).
\end{equation*}
\end{corollary}

\begin{remark}
Notice that, at leading order 
\begin{equation*}
xP_n(x;z) \sim P_{n+1}(x;z) + \frac{z^2}{4}P_{n-1}(x;z), \quad n\rightarrow
\infty,
\end{equation*}
and therefore 
\begin{equation*}
P_n(x;z) \sim z^n U_n\biggl(\frac{x}{z}\biggr), \quad n\rightarrow \infty,
\end{equation*}
where $(U_n)_{n\geq0}$ denotes the sequence of monic Chebyshev polynomials
of the second kind.

%defined by $U_{-1}(x) = 0$, $U_0(x)=1$ and
%\begin{equation*}
%    xU_n(x) = U_{n+1}(x)+\frac{1}{4}U_{n-1}(x), \quad n \geq 1.
%  \end{equation*}
In particular, denoting by $x_{n,k}$ the zeros of the polynomials $P_n(x;z)$%
, it follows that 
\begin{equation*}
x_{n,k}\sim z\cos\biggl(\frac{k\pi}{n+1}\biggr), \quad 1\leq k \leq n, \quad
n \rightarrow \infty.
\end{equation*}
\end{remark}

\section{Structure relation and differential equation}

It is well known that semiclassical polynomials are solutions of a linear
ODE with polynomial coefficients. In this section we find the ODE (in $x$)
satisfied by $P_n(x; z).$

\subsection{Structure relation}

An essential characteristic of semiclassical MOPS is the representation of $%
\partial_xP_n$ in terms of a finite number of terms of the sequence $%
(P_n)_{n\geq0}$, which is known in the literature as first structure relation, see \cite{Maroni91}.

\begin{theorem}
The polynomials $P_n(x;z)$ satisfy the differential-difference relation

\begin{equation}
\phi(x;z)\partial_x P_{n+1} (x)= (n+1)P_{n+2}(x)
+B_nP_{n}(x)+C_nP_{n-2}(x)+D_nP_{n-4}(x),  \label{eq:structure}
\end{equation}
where \flushleft
\begin{align}
B_n &= 4\gamma_{n+1}\biggl[\gamma_{n+2}\biggl(\gamma_{n+3}+\gamma_{n+2}+%
\gamma_{n+1}+\gamma_n-z^2\biggr) +\gamma_{n+1}\biggl(\gamma_{n+2}+%
\gamma_{n+1}+\gamma_n-z^2\biggr)  \notag \\
&\quad+\gamma_{n}\biggl(\gamma_{n+1}+\gamma_n+\gamma_{n-1}-z^2\biggr)-\frac{1%
}{2}-\frac{n}{4}\biggr],  \label{eq:Bn}
\end{align}
\begin{equation}
C_n=
4\gamma_{n+1}\gamma_{n}\gamma_{n-1}(\gamma_{n+2}+\gamma_{n+1}+\gamma_{n}+%
\gamma_{n-1}+\gamma_{n-2}-z^2),  \label{eq:Cn}
\end{equation}
and 
\begin{equation}
D_n = 4\gamma_{n+1}\gamma_{n}\gamma_{n-1}\gamma_{n-2}\gamma_{n-3}.
\label{eq:Dn}
\end{equation}
\end{theorem}

\begin{proof}
Here we follow the approach given in \cite{VAssche2018} \color{black}. From
the expansion of the polynomial $\phi(x;z)\partial_x P_{n+1}(x)$ in the
polynomial basis $(P_k)_{0\leq k\leq n+2 }$, we have 
\begin{equation*}
\phi(x;z)\partial_x P_{n+1}(x) = \displaystyle\sum_{k=0}^{n+2} d_{n,k} P_k
(x).
\end{equation*}
Since the linear functional $\bm{u}$ is symmetric and $\phi(x;z)$ is an even
polynomial in $x$, we have 
\begin{equation}
h_kd_{n,k} = \bigl<\bm{u}, \phi \partial_x P_{n+1} P_k \bigr> = 0, \quad
n+1\equiv k \text{ mod(2)}.  \label{eq:dnkstructure}
\end{equation}

From the Pearson's equation (\ref{eq:Pearson}) we get 
\begin{align*}
h_kd_{n,k} &= \bigl< \bm{u},\phi (\partial_x P_{n+1})P_k\bigr> = \bigl< %
\bm{u},\phi\partial_x(P_{n+1}P_k)\bigr>-\bigl< \bm{u}, \phi P_{n+1}
\partial_x P_k \bigr> \\
& =\bigl< \bm{u},\psi P_{n+1}P_k\bigr> - \bigl< \bm{u},\phi
P_{n+1}\partial_x P_k\bigr>.
\end{align*}
Since $(P_n)_{n\geq 0}$ is a MOPS, we conclude that $d_{n,k}=0$ for all $%
0\leq k < n-4$, and since we work with monic polynomials we have $d_{n,n+2}
= n+1$. Hence, we obtain (\ref{eq:structure}).\newline

Taking derivatives in (\ref{eq:3trr}), multiplying the resulting expression
by $\phi(x;z)$ and then using (\ref{eq:structure}) to replace the
derivatives we get 
\begin{align*}
& x^2P_{n+1}(x) -z^2P_{n+1}(x) +x\bigl[(n+1)P_{n+2}(x)
+B_nP_{n}(x)+C_nP_{n-2}(x)+D_nP_{n-4}(x)\bigr] \\
&= (n+3)P_{n+3}(x) +B_{n+1}P_{n+1}(x)+C_{n+1}P_{n-1}(x)+D_{n+1}P_{n-3}(x) \\
&\quad +\gamma_{n+1}\bigl[nP_{n+1}(x)
+B_{n-1}P_{n-1}(x)+C_{n-1}P_{n-3}(x)+D_{n-1}P_{n-5}(x)\bigr].
\end{align*}

From (\ref{eq:3trrSn}) and (\ref{eq:3trrSn2}) one finds 
\begin{align*}
&P_{n+3} (x)+ (\gamma_{n+2} + \gamma_{n+1}) P_{n+1} (x)+ \gamma_n
\gamma_{n+1} P_{n-1} (x)- z^2 P_{n+1}(x) \\
&\quad + (n+1)(P_{n+3} (x)+ \gamma_{n+2} P_{n+1}(x)) + B_n(P_{n+1}(x) +
\gamma_{n} P_{n-1}(x)) \\
&\quad + C_n(P_{n-1}(x) + \gamma_{n-2} P_{n-3}(x)) + D_n(P_{n-3}(x) +
\gamma_{n-4} P_{n-5}(x)) \\
&= (n+2) P_{n+3}(x) + B_{n+1} P_{n+1}(x) + C_{n+1} P_{n-1}(x) + D_{n+1}
P_{n-3}(x) \\
&\quad + \gamma_{n+1} \left( n P_{n+1}(x) + B_{n-1} P_{n-1}(x) + C_{n-1}
P_{n-3}(x) + D_{n-1} P_{n-5}(x) \right).
\end{align*}
Comparing the coefficients of $P_{n+1}$, $P_{n-1}$, $P_{n-3}$ and $P_{n-5}$
we get 
\begin{equation}
P_{n+1} \Longrightarrow \gamma_{n+2}+\gamma_{n+1} -z^2 + \gamma_{n+2}(n+1) +
B_n = B_{n+1}+n\gamma_{n+1},  \label{eq:AnBn}
\end{equation}
\begin{equation}
P_{n-1}\Longrightarrow \gamma_{n+1}\gamma_{n}+B_n\gamma_{n}+C_n =
C_{n+1}+\gamma_{n+1}B_{n-1},  \label{eq:BnCn}
\end{equation}
\begin{equation}
P_{n-3}\Longrightarrow C_n\gamma_{n-2} +D_n = D_{n+1}+C_{n-1}\gamma_{n+1},
\label{eq:CnDn}
\end{equation}
\begin{equation}
P_{n-5} \Longrightarrow D_n\gamma_{n-4} = \gamma_nD_{n-1}.
\label{eq:WalterDn}
\end{equation}
On the other hand, from (\ref{eq:WalterDn}) we deduce 
\begin{equation*}
\frac{D_n}{\gamma_{n+1}\gamma_{n}\gamma_{n-1}\gamma_{n-2}\gamma_{n-3}} = 
\frac{D_{n-1}}{\gamma_{n}\gamma_{n-1}\gamma_{n-2}\gamma_{n-3}\gamma_{n-4}},
\end{equation*}
i. e., $D_n / \gamma_{n+1}\gamma_{n}\gamma_{n-1}\gamma_{n-2}\gamma_{n-3}$ is
a constant. As a consequence, 
\begin{equation}
D_n = \kappa \gamma_{n+1}\gamma_{n}\gamma_{n-1}\gamma_{n-2}\gamma_{n-3},
\quad n\geq 4 ,  \label{eq:D4}
\end{equation}
for some nonzero constant $\kappa$. \newline
\newline
To find $\kappa$ we use (\ref{eq:structure}) and integrate with $P_0(x)=1$ 
\begin{equation*}
h_0 D_4 = \bigl< \bm{u},\phi \partial_x P_5\bigr>.
\end{equation*}
An integration by parts and using the Pearson's equation (\ref{eq:Pearson}) 
\begin{equation*}
h_0D_4 = \bigl< \bm{u},\psi P_5\bigr> = 4h_5.
\end{equation*}
Since $\gamma_{n}=h_n/h_{n-1}$ for $n\geq 1$ we obtain $D_4 =
4\gamma_5\gamma_4\gamma_3\gamma_2\gamma_1$. By direct comparison with (\ref%
{eq:D4}) for $n=4$ we conclude that $\kappa = 4$.\newline
\newline
Now replacing (\ref{eq:D4}) in (\ref{eq:CnDn}) we have 
\begin{equation*}
\frac{C_n}{\gamma_{n+1}\gamma_n\gamma_{n-1}}-\frac{C_{n-1}}{%
\gamma_n\gamma_{n-1}\gamma_{n-2}} = 4\left(\gamma_{n+2}-\gamma_{n-3}\right).
\end{equation*}
Summing from $3$ to $n$ gives 
\begin{equation*}
\frac{C_n}{\gamma_{n+1}\gamma_{n}\gamma_{n-1}}-\frac{C_2}{%
\gamma_3\gamma_2\gamma_1} =
4\left(\gamma_{n+2}+\gamma_{n+1}+\gamma_{n}+\gamma_{n-1}+\gamma_{n-2}-%
\gamma_1-\gamma_2-\gamma_3-\gamma_4\right).
\end{equation*}
Making use of (\ref{eq:structure}) and (\ref{eq:gammah_m}) we see that,
after integration by parts, 
\begin{equation*}
C_2 = 4\gamma_3\gamma_2\gamma_1(\gamma_4+\gamma_3+\gamma_2+\gamma_1-z^2).
\end{equation*}
Then (\ref{eq:Cn}) holds for $n\geq 3$. Similarly, using (\ref{eq:Cn}) in (%
\ref{eq:BnCn}), we get 
\begin{align*}
&\frac{B_n}{\gamma_{n+1}}-\frac{B_{n-1}}{\gamma_n} = 4\biggl[\gamma_{n+2}%
\biggl(\gamma_{n+3}+\gamma_{n+2}+\gamma_{n+1}+\gamma_{n}-z^2\biggr) %
-\gamma_{n-1}\biggl(\gamma_{n+1}+\gamma_n+\gamma_{n-1}+\gamma_{n-2}-z^2%
\biggr)-\frac{1}{4}\biggr],
\end{align*}
and summing from $1$ to $n$, we obtain 
\begin{align*}
\frac{B_n}{\gamma_{n+1}}-\frac{B_0}{\gamma_1} &= 4\biggl[\gamma_{n+2}\biggl(%
\gamma_{n+3}+\gamma_{n+2}+\gamma_{n+1}+\gamma_n-z^2\biggr) +\gamma_{n+1}%
\biggl(\gamma_{n+2}+\gamma_{n+1}+\gamma_n-z^2\biggr) \\
&\quad+\gamma_{n}\biggl(\gamma_{n+1}+\gamma_n+\gamma_{n-1}-z^2\biggr) %
-\gamma_{1}(\gamma_1+\gamma_2-z^2) -\gamma_2(\gamma_3+\gamma_2+\gamma_1-z^2)-%
\frac{n}{4}\biggr],
\end{align*}
% \begin{align*}
%     B_1 = 4\gamma_2\gamma_3(\gamma_1+\gamma_2+\gamma_3+\gamma_4 -z^2) +\frac{\gamma_2}{\gamma_1}B_0-\gamma_2
% \end{align*}
% and
% \begin{align*}
% B_n &= 4 \gamma_{n+1} \left( \gamma_{n+2} (\gamma_{n+1}+\gamma_{n+2}+\gamma_{n+3})+ \gamma_{n-1} \gamma_{n}+ (\gamma_{n}+\gamma_{n+1}+\gamma_{n+2}) (\gamma_{n}+\gamma_{n+1}-z^2) \right. \\
% & \quad \left. -\gamma_{3} (\gamma_{2}+\gamma_{3}+\gamma_{4})   - (\gamma_{1}+\gamma_{2}+\gamma_{3}) (\gamma_{1}+\gamma_{2}-z^2) + \frac{1}{4}-\frac{n}{4} \right)+\frac{\gamma_{n+1}}{\gamma_2}B_1
% \end{align*}
for $n\geq 1$.\newline
\newline
From (\ref{eq:gammah_m}) and (\ref{eq:structure}) we have

\begin{equation*}
B_0 = 4\gamma_1\left[\gamma_1(\gamma_2+\gamma_1-z^2)+\gamma_2(\gamma_3+%
\gamma_2+\gamma_1-z^2)\right],
\end{equation*}
and then we obtain (\ref{eq:Bn}).
\end{proof}

\begin{remark}
Notice that (\ref{eq:AnBn}) is equivalent to (\ref{eq:FreudLaguerre}) with $%
n\rightarrow n+1$.
\end{remark}

\subsection{Second order differential equation}

Let $(P_n)_{n\geq 0}$ be a MOPS with respect to a weight function $\omega(x)
= exp(-v(x))$ supported on an interval $[a,b]\subset \mathbb{R}$, either
bounded or unbounded, such that

\begin{equation*}
\displaystyle\int_{a}^b P_n(x)P_m(x) \omega(x) dx = h_n \delta_{n,m}.
\end{equation*}

Then the polynomials $P_n(x)$ satisfy the general TTRR (\ref{eq:3trr}).
Moreover, it has been shown, see \cite{Bauldry90,BonanClark90} and also \cite%
{ChenIsmail97,IsmailWimp98}, that under certain assumptions on $\omega(x)$
the polynomials $P_n$ satisfy the differential-difference relation

\begin{equation*}
P^{\prime }_n(x) = A_n(x)P_{n-1}(x)-B_n(x)P_n(x),
\end{equation*}
where $A_n(x)$ and $B_n(x)$ are explicitly given in terms of $\omega(x)$ and
the evaluation of $P_n$ at the end points $a$ and $b$. In particular, if $%
\omega(x)$ does not vanish at the endpoints of the interval of
orthogonality, then $A_n(x)$ and $B_n(x)$ take the form \cite{ChenIsmail97} 
\begin{equation}
A_n(x) = \frac{\omega(b^-)P^2_n(b^-)}{h_{n-1}(b-x)}+\frac{%
\omega(a^+)P^2_n(a^+)}{h_{n-1}(x-a)}+\displaystyle\int_{a}^b \frac{v^{\prime
}(x)-v^{\prime }(y)}{h_{n-1}(x-y)}P^2_n(y)w(y)dy,  \label{eq:AnIsm}
\end{equation}
and 
\begin{align}
B_n(x) = \frac{\omega(b^-)P_n(b^-)P_{n-1}(b^-)}{h_{n-1}(b-x)}&+\frac{%
\omega(a^+)P_n(a^+)P_{n-1}(a^+)}{h_{n-1}(x-a)}  \notag \\
&+\displaystyle\int_{a}^b \frac{v^{\prime }(x)-v^{\prime }(y)}{h_{n-1}(x-y)}%
P_n(y)P_{n-1}(y)w(y)dy.  \label{eq:BnIsm}
\end{align}
Notice that in \cite{ChenIsmail97} the authors proved the result for
orthonormal polynomials whilst (\ref{eq:AnIsm}) and (\ref{eq:BnIsm}) are for
monic orthogonal polynomials. For the sequence $\left(P_n(x;z)\right)_{n\geq
0}$ we have the following result.

\begin{proposition}
The polynomials $P_n(x;z)$ satisfy the differential-difference recurrence
relation 
\begin{equation}
\phi(x;z)\partial_x P_n(x;z) = 4\gamma_n(z)\mathcal{A}_n(x;z)P_{n-1}(x;z)-%
\mathcal{B}_n(x;z)P_{n}(x;z),  \label{eq:diffAnBnFreud}
\end{equation}
where 
\begin{align}
\mathcal{A}_n(x;z) = & \gamma_{n+1}(z)
\left(\gamma_{n+2}(z)+\gamma_{n+1}(z)+\gamma_{n}(z)\right)  \notag \\
& +\gamma_n(z)(\gamma_{n+1}(z)+\gamma_n(z)+\gamma_{n-1}(z))
\label{eq:AnFreud} \\
& + \phi(x;z)(x^2+\gamma_{n+1}(z)+\gamma_n(z))-\frac{n}{2}-\frac{1}{4} , 
\notag
\end{align}
and 
\begin{equation}
\mathcal{B}_n(x;z) = x\left[4\gamma_n(z)(\gamma_{n+1}(z)+\gamma_n(z)+%
\gamma_{n-1}(z)+\phi(x;z))-n\right].  \label{eq:BnFreud}
\end{equation}
Moreover, 
\begin{equation}
\mathcal{B}_n(x;z) + \mathcal{B}_{n+1}(x;z) =4x(\mathcal{A}_n(x;z)
-x^2\phi(x;z)).  \label{eq:BnAnFreud}
\end{equation}
\end{proposition}

\begin{proof}
Let $h_n=\bigl<\bm{u},P^2_n\bigr>$ and let $\omega(x) = e^{-v(x)}$ with $%
v(x)=x^4$ be the weight function associated with the linear functional $%
\bm{u}$ defined in (\ref{eq:FreudLF}). Under the assumptions of \cite%
{ChenIsmail97} we can write

\begin{equation}
\partial_xP_n(x;z) = A_n(x)P_{n-1}(x;z)-B_n(x)P_n(x;z),  \label{eq:diffAnBn}
\end{equation}
where $A_n(x)$ and $B_n(x)$ are given by (\ref{eq:AnIsm}) and (\ref{eq:BnIsm}%
), respectively. \newline

From (\ref{eq:symmetry}) we have 
\begin{equation}
A_n(x) = -\frac{2zP^2_{n}(z;z)e^{-z^4}}{h_{n-1}} + \frac{4}{h_{n-1}}%
\displaystyle\int_{-z}^z (x^2+xy+y^2)P^2_{n}(y;z) e^{-y^4}dy,
\label{eq:AnSym}
\end{equation}
and 
\begin{equation}
B_n(x) = -\frac{2xP_{n-1}(z;z)P_{n}(z;z)e^{-z^4}}{h_{n-1}\phi}+\frac{4}{%
h_{n-1}}\displaystyle\int _{-z}^z(x^2+xy+y^2)P_{n-1}(y;z)P_{n}(y;z)
e^{-y^4}dy.  \label{eq:BnSym}
\end{equation}
From (\ref{eq:Pn(z;z)^2}), (\ref{eq:3trrSn}), (\ref{eq:3trrSn2}) and since $%
P_n(x;z)$ is orthogonal with respect to the linear functional $\bm{u}$ we
get 
\begin{equation*}
A_n(x) = \frac{4\left[h_{n+2}+(\gamma_n+\gamma_{n+1})^2h_n+\gamma^2_n%
\gamma_{n-1}h_{n-1}\right]-(2n+1)h_n}{\phi h_{n-1}}+\frac{%
4\left(x^2h_n+h_{n+1}+\gamma^2_nh_{n-1}\right)}{h_{n-1}}.
\end{equation*}
Similarly, from (\ref{eq:Pn(z;z)}), (\ref{eq:3trrSn}) and (\ref{eq:3trrSn2})
we obtain 
\begin{equation*}
B_n(x) =\frac{2x}{\phi}\left[2\gamma_n(\gamma_{n+1}+\gamma_n+\gamma_{n-1})-%
\frac{n}{2}\right]+ 4x\frac{ h_n}{h_{n-1}}.
\end{equation*}
Using (\ref{eq:gammah_m}), multiplying (\ref{eq:diffAnBn}) by $\phi$ and
arranging terms we get (\ref{eq:diffAnBnFreud}). \smallskip

On the other hand, from (\ref{eq:3trrFreud}) and (\ref{eq:gammah_m}) 
\begin{equation*}
\frac{P_{n-1}(x)}{h_{n-1}}+\frac{P_{n+1}(x)}{h_n} = \frac{xP_n(x)}{%
\gamma_{n}h_{n-1}}.
\end{equation*}
Then from (\ref{eq:BnSym}) we have

\begin{equation*}
B_n(x) + B_{n+1}(x) = \frac{-2xzP^2_n(z;z)e^{-z^4}}{\phi\gamma_{n} h_{n-1}}+%
\frac{4}{\gamma_{n}h_{n-1}} \displaystyle%
\int_{-z}^z(x^2+xy+y^2)yP^2_n(y;z)e^{-y^4}dy,
\end{equation*}
while from (\ref{eq:AnSym}) we get 
\begin{equation*}
B_n(x) + B_{n+1}(x) = \frac{x}{\gamma_{n}}A_n(x)+\frac{4}{\gamma_{n}h_{n-1}} %
\displaystyle\int_{-z}^z(x^2+xy+y^2)(y-x)P^2_n(y;z)e^{-y^4}dy.
\end{equation*}
Since $4(x^2+xy+y^2)(y-x)=4y^3-4x^3=v^{\prime }(y)-v^{\prime }(x)$,
integration by parts and (\ref{eq:gammah_m}) yield 
\begin{equation*}
B_n(x)+B_{n+1}(x) = \frac{x}{\gamma_{n}}A_n(x)-4x^3.
\end{equation*}
Multiplying by $\phi(x;z)$ (\ref{eq:BnAnFreud}) follows.
\end{proof}

\begin{corollary}
Let the differential operator $L_{1,n}$ be defined by 
\begin{equation}
L_{1,n} = \phi(x;z)\partial_x+\mathcal{B}_n(x;z), \quad n\in \mathbb{N}.
\label{eq:L1}
\end{equation}
Then 
\begin{equation}
L_{1,n}P_n(x;z) = 4\gamma_n(z)\mathcal{A}_n(x;z)P_{n-1}(x;z).
\label{eq:L1Pn}
\end{equation}
On the other hand, using (\ref{eq:3trrFreud}) and (\ref{eq:BnAnFreud}) in (%
\ref{eq:L1Pn}) we define the differential operator $L_{2,n}$ as 
\begin{equation}
L_{2,n} = -\phi(x;z)\partial_x + \mathcal{B}_{n}(x;z)+4x^3\phi(x;z).
\label{eq:L2}
\end{equation}
Then 
\begin{equation}
L_{2,n} P_{n-1}(x;z) = 4\mathcal{A}_{n-1}(x;z)P_{n}(x;z), \quad n\in\mathbb{N%
}.  \label{eq:L2Pn}
\end{equation}
\end{corollary}

\begin{remark}
\label{remark3} Notice that $\mathcal{A}_n(x;z)$ in (\ref{eq:AnFreud}) is a
monic polynomial of degree $4$ 
\begin{equation}
\mathcal{A}_n(x;z) = x^4+b_n(z)x^2+c_n(z), \label{bicuartica}
\end{equation}
where 
\begin{equation*}
b_n(z) = \gamma_{n+1}(z)+\gamma_n(z)-z^2,
\end{equation*}
and 
\begin{equation*}
c_n(z) = \gamma_n(z) (\gamma_{n+1}(z) + \gamma_{n}(z)+\gamma_{n-1}(z)-z^2) +
\gamma_{n+1}(z) ( \gamma_{n+2}(z) + \gamma_{n+1}(z)+\gamma_n(z)-z^2) -\frac{n%
}{2}-\frac{1}{4}.
\end{equation*}
Then, for $n\geq 1$, the polynomial $\mathcal{A}_n(x;z)$ has two real zeros
and two conjugated complex zeros.

\begin{equation}
\eta_{1,2}(n,z) = \pm \sqrt{\frac{1}{2}\left(\sqrt{b^2_n(z)-4c_n(z)}-b_n(z)\right)}, \label{etas}
\end{equation}
and 
\begin{equation}
\zeta_{1,2}(n,z) = \pm i\sqrt{\frac{1}{2}\left(\sqrt{b^2_n(z)-4c_n(z)}+b_n(z)\right)}. \label{zetas}
\end{equation}
\end{remark}

As a consequence of (\ref{eq:diffAnBnFreud}), we can now derive a second
order linear differential equation in $x$ for the polynomials $P_n(x;z)$.

\begin{theorem}[Holonomic differential equation]
The polynomials $P_{n}(x;z)$ are solution of the second order linear
differential equation 
\begin{equation}
\partial^2_xP_n(x;z)+R(x;n)\partial_xP_n(x;z)+S(x;n)P_n(x;z) = 0,
\label{eq:PnODE}
\end{equation}
\end{theorem}

where $R(x;n)$ and $S(x;n)$ are the rational functions given by 
\begin{equation*}
R(x;n) = -2x\frac{(2x^2\phi(x;z)-1)\mathcal{A}_n(x;z)+\phi(x;z)(%
\gamma_{n+1}(z)+\gamma_n(z)+2x^2-z^2)}{\phi(x;z)\mathcal{A}_n(x;z)},
\end{equation*}
and 
\begin{align*}
S(x;n) &= \frac{1}{\phi^2\mathcal{A}_n(x;z)} \Biggl[ -\mathcal{A}_n(x;z) 
\mathcal{B}_n(x;z) \left(\mathcal{B}_{n}(x;z) + 4 \phi(x;z) x^3\right) \\
&\quad - 2 x\mathcal{B}_n(x;z) \phi(x;z) \left(2 x^2 - z^2 + \gamma_{n}(z) +
\gamma_{n+1}(z)\right) \\
&\quad + \mathcal{A}_n(x;z) \phi(x;z) \Biggl(-n + 4 \gamma_{n}(z) \left(3
x^2 - z^2 + \gamma_{n-1}(z) + \gamma_{n}(z) + \gamma_{n+1}(z)\right)\Biggr)
\\
&\quad - 16\gamma_n(z)\mathcal{A}^2_n(x;z)\mathcal{A}_{n-1}(x;z) \Biggr].
\end{align*}

\begin{proof}
From (\ref{eq:L1Pn}) and (\ref{eq:L2Pn}) the polynomials $P_n(x;z)$ satisfy
the factored equation (for further details see \cite{ChenIsmail97}) 
\begin{equation}
L_{2,n}\left(\frac{1}{\mathcal{A}_n}\left(L_{1,n}P_n\right)\right) =
16\gamma_n\mathcal{A}_{n-1}P_n.  \label{eq:L2nL1nPn}
\end{equation}
On the other hand, from (\ref{eq:L1}) and (\ref{eq:L2}), we have 
\begin{align}
& L_{2,n}\left(\frac{1}{\mathcal{A}_n}\left(L_{1,n}P_n\right)\right) =
L_{2,n}\left(\frac{1}{\mathcal{A}_n}\left(\phi P^{\prime }_n+\mathcal{B}%
_nP_n\right)\right)  \notag \\
\notag \\
&= -\frac{\phi^2}{\mathcal{A}_n}\partial^2_x + \phi\left[\frac{\mathcal{A}%
^{\prime }_n}{\mathcal{A}^2_n}\phi+\frac{4x^3\phi}{\mathcal{A}_n}-\frac{%
\phi^{\prime }}{\mathcal{A}_n}\right]\partial_xP_n + \left[-\frac{\phi 
\mathcal{B}^{\prime }_n}{\mathcal{A}_n}+\phi \mathcal{B}_n\frac{\mathcal{A}%
^{\prime }_n}{\mathcal{A}^2_n}+\frac{\mathcal{B}_n}{\mathcal{A}_n}%
\left(4x^3\phi+\mathcal{B}_n\right)\right]P_n.  \label{eq:L2(L1Pn)}
\end{align}
Equating (\ref{eq:L2nL1nPn}) and (\ref{eq:L2(L1Pn)}) and multiplying the
resulting expression by $-\mathcal{A}_n/\phi^2$ we obtain

\begin{equation*}
R(x;n) = \frac{1}{\phi}\left[\phi^{\prime 3}\phi-\frac{\mathcal{A}^{\prime
}_n\phi}{\mathcal{A}_n}\right],
\end{equation*}
and 
\begin{equation*}
S(x;n) = \frac{1}{\phi^2}\left[\phi \mathcal{B}^{\prime }_n-\phi\mathcal{B}_n%
\frac{\mathcal{A}^{\prime }_n}{\mathcal{A}_n}-\mathcal{B}_n(4x^3\phi+%
\mathcal{B}_n)-16\gamma_n\mathcal{A}_n\mathcal{A}_{n-1}\right].
\end{equation*}
Thus, (\ref{eq:AnFreud}) and (\ref{eq:BnFreud}) yield the result.
\end{proof}

\section{Electrostatic interpretation}
\label{sect-Electr-Int}

It is very well known that the zeros of orthogonal polynomials with respect
to a positive definite linear functional are real, simple, and located in
the interior of the convex hull of the support of the positive Borel measure
of orthogonality, see \cite{Chihara78, ArdilaMarcellan21}. Thus, let $%
\{x_{n,k}(z)\}_{1\leq k \leq n}$ denote the zeros of $P_n(x)$ in an
increasing order, i. e., 
\begin{equation}
P_n(x_{n,k}(z)) = 0, \quad \text{for all } 1\leq k\leq n,  \label{eq:Pn(xnk)}
\end{equation}
and $x_{n,1}(z)<x_{n,2}(z)<\cdots< x_{n,n}(z) $. \newline

Next we will focus our attention on an electrostatic interpretation of the
zeros of the polynomial $P_n(x;z)$. To do that, let us consider a system of $%
n$ movable unit charges at positions $X=(x_1,x_2,...,x_n)$ and suppose that
the charges interact with each other under the presence of potential $V$.
Then the total energy of the system, $E(X)$, is 
\begin{equation}
E(X) = \displaystyle\sum^n_{k=1} V(x_k) -2\displaystyle\sum_{1\leq j<k\leq
n} \ln|x_j-x_k|.  \label{eq:E(X)}
\end{equation}
% The electrostatic equilibrium occurs in stationary points $X^*=(x^*_1, x^*_2,...,x^*_n)$ of the total energy $E(X)$ such that
% \begin{equation}
%     \frac{\partial E(X)}{\partial x_k}\biggr|_{X=X^*} = -\displaystyle\sum_{1\leq j<k\leq n} \frac{2}{x^*_k-x^*_j}+V'(x_k) =0,
%     \label{eq:Equilibrium}
% \end{equation}
% and this configuration is precisely the zeros of the corresponding orthogonal polynomials.
% In [Electrostatic Ismail] it is stated that if $\omega(x)>0$ on $(a,b)$, both $v(x)$ and $\ln(A_n)$ for $A_n(x)$ in (\ref{eq:AnIsm}) are twice continuously differentiable function on $(a,b)$, then the total energy $E(\bm{x})$ of a system of $n$ unit movable unit charged in $[a,b]$ at positions $\bm{x}=(\Tilde{x}_1,...,\Tilde{x}_n)$ under the influence of a logarithmic potential
% \begin{equation*}
%     E(\bm{x}) = \displaystyle\sum^n_{k=1} V(\Tilde{x}_k) -2\displaystyle\sum_{1\leq j<k\leq n} \ln|\Tilde{x}_j-\Tilde{x}_k|
% \end{equation*}
% has a unique point of global minimum which is precisely the zeros of the orthogonal polynomials $p_n$, i.e.

%       \begin{equation}
%        \left.\frac{\partial E}{\partial \Tilde{x}_k}\right|_{\bm{x}=x_{n,k}} = \displaystyle\sum^n_{k=1} V'(\Tilde{x}_{k}) -2\displaystyle\sum_{1\leq j<k\leq n} \frac{1}{{x}_{n,k}-{x}_{n,j}}=0, \quad 1\leq k\leq n
%        \label{eq:electroequilibrium}
%    \end{equation}

\begin{theorem}
The zeros of $P_n(x;z)$ are located at the equilibrium points of $n$ unit
charged particles located in the interval $(-z, z)$ under the influence of
external the potential

\begin{equation*}
V_n(x;z) = x^4 -\ln\lvert x^2-z^2\rvert +\ln\lvert
x^2-\zeta^2(n,z)\rvert+\ln\lvert x^2-\eta^2(n,z)\rvert,
\end{equation*}
where $\zeta(n,z)$ and $\eta(n,z)$ are given in Remark \ref{remark3}. 
\newline
\end{theorem}

\begin{proof}
If 
\begin{equation*}
P_n(x;z) = \displaystyle\prod_{k=1}^n(x-x_{n,k}),
\end{equation*}
then \cite[Chapter 10]{ArdilaMarcellan21} 
\begin{equation}
\left(\frac{\partial^2_xP_n}{\partial_xP_n}\right)_{x=x_{n,k}} = -%
\displaystyle\sum_{1\leq j<k\leq n} \frac{2}{x_{n,k}-x_{n,j}}.
\label{eq:DDPxnk}
\end{equation}
Evaluating $(\ref{eq:PnODE})$ at $x_{n,k}$ we get 
\begin{equation}
\left(\frac{\partial^2_xP_n}{\partial_xP_n}\right)_{x=x_{n,k}} = 4x_{n,k}^3+%
\frac{\mathcal{A}_n^{\prime }(x_{n,k};z)}{\mathcal{A}_n(x_{n,k};z)}-\frac{%
\phi^{\prime }(x_{n,k};z)}{\phi(x_{n,k};z)}.  \label{eq:DDP(xnk;z)}
\end{equation}
Using (\ref{eq:DDP(xnk;z)}) in (\ref{eq:DDPxnk}) we have 
\begin{equation}
\displaystyle\sum_{1\leq j<k\leq n} \frac{2}{x_{n,k}-x_{n,j}} + 4x_{n,k}^3+%
\frac{\mathcal{A}_n^{\prime }(x_{n,k};z)}{\mathcal{A}_n(x_{n,k};z)}- \frac{%
\phi^{\prime }(x_{n,k};z)}{\phi(x_{n,k};z)} = 0.  \label{eq:GE(xnk)}
\end{equation}
% Comparing () and () we conclude that
%  \begin{equation*}
%      V_n(x;z) = x^4+\ln\left(\frac{\mathcal{A}_n(x;z)}{\phi(x;z)} \right) = 0
%  \end{equation*}
Notice that (\ref{eq:GE(xnk)}) is the gradient of (\ref{eq:E(X)}), at the
set of points $X^*=(x_{n,1},x_{n,2},...,x_{n,n})$, associated with the
external potential 
\begin{equation}
V_n(x;z) = x^4 + \ln \biggl(\frac{\mathcal{A}_n(x;z)}{\phi(x;z)}\biggr).  \label{eq:TotalPot}
\end{equation}
This means that the zeros of the truncated Freud polynomials $P_n(x;z)$ are critical points of the energy functional.
\end{proof}

\medskip

\begin{remark}
Notice that the total external potential has two terms 
\begin{equation}
V_n(x;z) = x^4+\ln\left(\frac{\mathcal{A}_n(x;z)}{\phi(x;z)} \right) = V_{%
\text{long}}(x) + V_{\text{n,short}}(x;z). \label{totalExtPoten}
\end{equation}
Following \cite{Ismail2000}, the long range potential is, as expected, $V_{%
\text{long}}(x) = x^4$ while the short range potential depends on the degree 
$n$ and is given by 
\begin{align*}
V_{n, \text{short}}(x;z) &= \ln\left(\frac{\mathcal{A}_n(x;z)}{\phi(x;z)}%
\right) \\
&= \ln \left| x-\eta_1(n,z)\right| + \ln \left| x-\eta_2(n,z)\right| +\ln
\left| x-\zeta_1(n,z)\right|+\ln \left| x-\zeta_2(n,z)\right| \\
&\quad - \ln \left| x-z\right|-\ln \left| x+z\right|.
\label{eq:Electromodel}
\end{align*}
We consider that the potential energy at $x$ of a point charge $q$ located at $t$ is $-q\ln\bigl|x-t\bigr|$. Then the external field is generated by two fixed charges $+1$ at $\pm z$ (due to a perturbation of the weight function) plus four fixed charges of magnitude $-1$; two of them at the real positions $\eta_1(n,z)$, $\eta_2(n,z)$, and the remaining ones at complex positions $\zeta_1(n,z)$ and $\zeta_2(n,z)$, according to remark \ref{remark3}.
\end{remark}

\section{The variable $z$}

In this section we will study $\bm{u}_{2n}$, $\gamma_n$ and the zeros of the polynomials $P_n(x;z)$ as functions of $z$.

\subsection{The moments}

\begin{proposition}
The even moments $\bm{u}_{2n} = \bigl<\bm{u},x^{2n}\bigr>$ related to the
linear functional $\bm{u}$ satisfy the following differential-recurrence
relations 
\begin{equation}
z\partial_z \bm{u}_{2n}(z) = (2n+1)\bm{u}_{2n}(z)-4\bm{u}_{2n+4}(z)
\label{eq:zupunto}
\end{equation}
and 
\begin{equation}
\partial_z\bm{u}_{2n+2}(z) = z^2\partial_z\bm{u}_{2n}(z).  \label{eq:u'}
\end{equation}
\end{proposition}

\begin{proof}
Setting $x=zt$ in \eqref{eq:evenmomentsTF}, we have 
%%%%%%%%%%%%%%%%%%%%%%%%%%%%%%%%%%%%%%%%%%%%%%%%%%%
\begin{equation}
\bm{u}_{2n}(z) = 2\displaystyle\int_{0}^z x^{2n}e^{-x^4}dx = 2z^{2n+1}%
\displaystyle\int_{0}^1 t^{2n}e^{-z^4t^4}dt.  \label{eq:moments(t)}
\end{equation}
Then 
\begin{align*}
\partial_z\bm{u}_{2n}(z) & = 2(2n+1)z^{2n}\displaystyle%
\int_{0}^1t^{2n}e^{-z^4t^4}dt -8z^{2n+4}\displaystyle%
\int_{0}^1t^{2n+4}e^{-z^4t^4}dt \\
& = \frac{(2n+1)}{z} \bm{u}_{2n}(z) -\frac{4}{z}\bm{u}_{2n+4}(z),
\end{align*}
which yields \eqref{eq:zupunto}. On the other hand, we have 
\begin{equation}
\partial_z\bigl<\bm{u},\phi x^{2n}\bigr> = \partial_z\bm{u}_{2n+2}-2z\bm{u}%
_{2n}-z^2\partial_z\bm{u}_{2n},  \label{eq:partialzphiu}
\end{equation}
where $\phi(x;z) = x^2-z^2$ was defined in (\ref{eq:PearsonFreud}).
Meanwhile, setting again $x=zt$, we get 
\begin{align*}
\partial_z\bigl<\bm{u},\phi x^{2n}\bigr> &= 2\partial_z\displaystyle%
\int_{0}^z \phi(x;z) x^{2n}e^{-x^4}dx \\
&= 2 \displaystyle\int_{0}^1 \partial_z\biggl(\phi(t;z)z^{2n+1} e^{-z^4t^4}%
\biggr) t^{2n} dt \\
& = 2\displaystyle\int_{0}^1\biggl((2n+3)z^{2n+2}-4z^{2n+6}t^4\biggr)%
t^{2n}(t^2-1)e^{-z^4t^4}dt.
\end{align*}
Taking into account \eqref{eq:moments(t)}, we obtain 
\begin{equation*}
\partial_z\bigl<\bm{u},\phi x^{2n}\bigr> = -\frac{1}{z}\biggl(4\bm{u}%
_{2n+6}-4z^2\bm{u}_{2n+4}-(2n+3)\bm{u}_{2n+2}+(2n+3)z^2\bm{u}_{2n}\biggr),
\end{equation*}
and making use of \eqref{eq:recumoment} one finds that 
\begin{equation}
\partial_z\bigl<\bm{u},\phi x^{2n}\bigr> = -2z\bm{u}_{2n}.
\label{eq:partialzphiu2}
\end{equation}
Equating \eqref{eq:partialzphiu} and \eqref{eq:partialzphiu2} yields %
\eqref{eq:u'}.
\end{proof}

\bigskip

Using the differential-recurrence relation (\ref{eq:u'}), we can obtain a
first order linear ODE in $z$ for the Stieltjes function $\mathbf{S}_{\bm{u}%
}(t;z)$.

\begin{proposition}
Let $\mathbf{S}_{\bm{u}}(t;z)$ be defined by (\ref{eq:StieltjesFreud}) and
let $\phi(x;z)$ be defined by (\ref{eq:PearsonFreud}). Then 
\begin{equation}
\phi(t;z)\partial_z\mathbf{S}_{\bm{u}}(t;z)= 2te^{-z^4}.
\label{eq:StieltjesFreud_z}
\end{equation}
\end{proposition}

\begin{proof}
From (\ref{eq:StieltjesFreud}) we see that 
\begin{equation*}
\displaystyle\sum_{n\geq0}\frac{\bm{u}_{2n+2}(z)}{t^{2n+1}} = \displaystyle%
\sum_{n\geq 1} \frac{\bm{u}_{2n}(z)}{t^{2n-1}} = t^2\biggl(\displaystyle%
\sum_{n\geq1} \frac{\bm{u}_{2n}(z)}{t^{2n+1}}\biggr) = t^2\biggl(\mathbf{S}_{%
\bm{u}}(t;z)-\frac{\bm{u}_0(z)}{t}\biggr).
\end{equation*}
Taking the derivative with respect to $z$ and using (\ref{eq:u'}), we obtain 
\begin{equation*}
z^2\partial_z\mathbf{S}_{\bm{u}}(t;z) = t^2\partial_z\mathbf{S}_{\bm{u}%
}(t;z)-t\partial_z\bm{u}_0(z).
\end{equation*}
From \eqref{eq:zupunto} and \eqref{eq:recumomentsexp}, we get 
\begin{equation*}
\partial_z\bm{u}_0(z) = 2e^{-z^4}.
\end{equation*}
\end{proof}

\begin{remark}
Notice that (\ref{eq:StieltjesFreud_z}) with the initial condition $\mathbf{S%
}_{\bm{u}}(t;0)=0$ yields 
\begin{equation*}
\mathbf{S}_{\bm{u}}(t;z) = 2t\displaystyle\int_{0}^z\frac{e^{-x^4}}{t^2-x^2}%
dx,
\end{equation*}
which is (\ref{eq:StieltjesFreud}), provided that 
\begin{equation*}
\displaystyle\sum_{k\geq0}\frac{x^{2n}}{t^{2n+1}}=\frac{t}{t^2-x^2}.
\end{equation*}
\end{remark}

\subsection{Toda-type behaviour}

Next, we will obtain a differential recurrence relation for the sequence $%
(\gamma_n(z))_{n\geq 0}$.

\begin{theorem}
The functions $h_n(z)$ and $\gamma_n(z)$ satisfy the Toda-type equations 
\begin{equation*}
\vartheta \ln h_n(z) =4\biggl[\frac{n}{2}+\frac{1}{4} - \gamma_n\bigl(%
\gamma_{n+1}+\gamma_n+\gamma_{n-1}\bigr)-\gamma_{n+1}\bigl(%
\gamma_{n+2}+\gamma_{n+1}+\gamma_{n}\bigr)\biggr]
\end{equation*}
and 
\begin{equation*}
\vartheta \ln \gamma_{n}(z) = 4\biggl[\gamma_{n-1}\bigl(\gamma_n+%
\gamma_{n-1}+\gamma_{n-2}\bigr)-\gamma_{n+1}(z)\bigl(\gamma_{n+2}+%
\gamma_{n+1}+\gamma_{n}\bigr)+\frac{1}{2}\biggr],
\end{equation*}
\end{theorem}

where $\vartheta$ denotes the differential operator 
\begin{equation*}
\vartheta = z\partial_z.
\end{equation*}

\begin{proof}
Following \cite{Lyu}, taking the derivative of $h_n = \langle \bm{u}%
,P_n^2\rangle$ with respect to $z$, we have 
\begin{equation*}
\partial_z\ln h_n = \frac{2P^2_n(z;z)e^{-z^4}}{h_n}.
\end{equation*}
% Note that the second term on the right vanishes from (\ref{eq:OPS}) since deg($\partial_z P_n$) $\leq n-2$.

Then, from (\ref{eq:Pn(z;z)^2}) we get 
\begin{equation}
\vartheta \ln h_n = 2n+1 - 4\biggl[\gamma_n\bigl(\gamma_{n+1}+\gamma_n+%
\gamma_{n-1}\bigr)+\gamma_{n+1}\bigl(\gamma_{n+2}+\gamma_{n+1}+\gamma_{n}%
\bigr)\biggr].  \label{eq:h'}
\end{equation}
Taking into account (\ref{eq:gammah_m}), we obtain 
\begin{equation*}
\ln(\gamma_n) = \ln(h_n)-\ln(h_{n-1}),
\end{equation*}
and then 
\begin{equation*}
\vartheta \ln \gamma_n = \vartheta \ln h_n-\vartheta \ln h_{n-1}.
\end{equation*}
Hence, from (\ref{eq:h'}) our statement follows.
\end{proof}

% \subsection{Power series}
% The Maclaurin series of the function $\gamma_{n}(z)$ is
% \begin{equation*}
%     \gamma_n(z) = \displaystyle\sum_{k=1}^\infty \tau_{n,k}z^{2k}
% \end{equation*}
% with
% \begin{equation*}
%     \tau_{n,1} = \frac{n^2}{4n^2-1}
% \end{equation*}
% and
% \begin{equation*}
%     \tau_{n,k} = \cdots, \quad k\geq 2
% \end{equation*}

% \begin{proof}
%     From () we have
%     \begin{equation*}
%         \frac{z}{2}\gamma'_n(z) = \displaystyle\sum_{k=1}^\infty k \tau_{n,k} z^{2k}
%     \end{equation*}
% \end{proof}
% \subsection{Nonlinear ODE}
% \begin{theorem}
%     The function $\gamma_{n}(z)$ satisfies
%     \begin{equation*}
%         \cdot
%     \end{equation*}
% \end{theorem}
% \begin{proof}
%    \begin{equation*}
%        z^2g_{n-1}^2 = \gamma_{n-1}(g_{n-1}+g_{n})(g_{n-1}+g_{n-2})
%    \end{equation*}
% \end{proof}

\subsection{Dynamical behaviour of zeros}

We are now interested to study the motion of zeros of the polynomials $%
P_n(x;z)$ in terms of the parameter $z.$ Here we follow the approach given
in \cite{IsmailMa2011}. \newline

Let $\bm{v}$ be the linear functional defined by 
\begin{equation*}
\bigl< \bm{v},p(x)\bigr> = \displaystyle\int_{-z^{3/4}}^{z^{3/4}}p(x)
e^{-zx^4} dx, \quad p \in \mathbb{P}, \quad z>0.
\end{equation*}
If $(Q_n(x;z))_{n\geq 0}$ is the sequence of monic orthogonal polynomials
with respect to $\bm{v}$, then it is easy to check that 
\begin{equation}
Q_n(x;z) = \frac{1}{z^{n/4}}P_n(z^{1/4}x;z), \quad n\geq 0.  \label{eq:QnPn}
\end{equation}
Moreover 
\begin{equation*}
x Q_n(x;z) = Q_{n+1}(x;z)+\biggl[\frac{\gamma_n(z)}{\sqrt{z}}\biggr]%
Q_{n-1}(x;z)  \label{eq:xQn}
\end{equation*}
and

\begin{equation}
x^2Q_n(x;z) = Q_{n+2}(x;z) + \biggl[\frac{\gamma_{n+1}(z)+\gamma_n(z)}{\sqrt{%
z}}\biggr]Q_n(x;z) +\biggl[\frac{\gamma_n(z)\gamma_{n-1}(z)}{z}\biggr] %
Q_{n-2}(x;z) .  \label{eq:x^2Qn}
\end{equation}
Also it is easy to show that 
\begin{equation}
h_n(z) =z^{\frac{2n+1}{4}}||Q_n||^2,  \label{eq:hn||Qn||}
\end{equation}
where $||Q_n||^2=\bigl< \bm{v},Q^2_n\bigr>$ is the squared norm of the
polynomials $Q_n(x;z)$.\newline

Notice that (\ref{eq:gammah_m}) and (\ref{eq:hn||Qn||}) yield 
\begin{equation}
\frac{||Q_n||^2}{||Q_{n-1}||^2} = \frac{\gamma_n(z)}{\sqrt{z}}.
\label{eq:||Qn||/||Qn-1||}
\end{equation}
In a next step we can obtain a differential relation in $z$ for the
polynomials $Q_n(x;z)$ and, as a consequence, for the polynomials $P_n(x;z)$.

\begin{proposition}
The polynomials $P_n(x;z)$ satisfy 
\begin{align}
& -nP_n(x;z) +x\partial_xP_n(x;z)+z\dot{P}_n(x;z)  \notag \\
& = 4\gamma_n(z)\gamma_{n-1}(z)\biggl[\gamma_{n+1}(z)+\gamma_n(z)+%
\gamma_{n-1}(z)+\gamma_{n-2}(z)\biggr]P_{n-2}(x;z)  \label{eq:Pnpunto} \\
& \quad +4\biggl[\gamma_n(z)\gamma_{n-1}(z)\gamma_{n-2}(z)\gamma_{n-3}(z)%
\biggr]P_{n-4}(x;z),  \notag
\end{align}
where the dot $<\cdot>$ means derivative with respect to the variable $z$.
\end{proposition}

% \begin{proposition}
%     The polynomials $Q_n(x;z)$ satisfy the differential property

%     \begin{equation}
%         \dot{Q}_n(x;z) = \frac{\gamma_n\gamma_{n-1}}{z^{3/2}}\biggl[\gamma_{n+1}+\gamma_n+\gamma_{n-1}+\gamma_{n-2}\biggr]Q_{n-2}(x;z)+\biggl[\frac{\gamma_n\gamma_{n-1}\gamma_{n-2}\gamma_{n-3}}{z^2}\biggr]Q_{n-4}(x;z).
%     \label{eq:Qpunto}
%     \end{equation}

%     where the dot $<\cdot>$ means derivative with respect to the variable $z$.
%     \end{proposition}

\begin{proof}
Since $\dot{Q}_n(x;z)$ is a polynomial of degree $n-2$ we can write

\begin{equation*}
\dot{Q}_n(x;z) = \displaystyle\sum_{k=0}^{n-2} r_{n,k}(z) Q_k(x;z),
\end{equation*}

where the coefficients $r_{n,k}(z)$ are given by (see \cite{IsmailMa2011,
WVAssche22})

\begin{equation}
r_{n,k}(z) = \frac{\bigl<\bm{v},x^4Q_nQ_k \big>}{||Q_k||^2} = \frac{1}{%
||Q_k||^2}\displaystyle\int_{-z^{3/4}}^{z^{3/4}} x^4
Q_n(x;z)Q_k(x;z)e^{-zx^4}dx.  \label{eq:rnk}
\end{equation}
Since the polynomials $Q_n(x;z)$ are orthogonal with respect to the linear
functional $\bm{v}$, we see that $r_{n,k}(z) = 0$ for all $0\leq k < n-4$.
Moreover, since $x^4$ is an even polynomial and the polynomials $Q_k(x;z)$
are symmetric we have $r_{n,k}(z) = 0$ if $n\equiv k \text{ mod(2)}$. Then
we conclude that 
\begin{equation*}
\dot{Q}_n(x;z) = r_{n,n-2}(z)Q_{n-2}(x;z) + r_{n,n-4}(z)Q_{n-4}(x;z).
\end{equation*}
From (\ref{eq:rnk}) and (\ref{eq:x^2Qn}), we have 
\begin{align}
||Q_{n-2}||^2 r_{n,n-2} = \bigl<\bm{v},x^4Q_nQ_{n-2}\bigr> = \frac{1}{\sqrt{z%
}}(\gamma_{n+1}+\gamma_n)||Q_n||^2+\frac{\gamma_n\gamma_{n-1}}{z^{3/2}}%
(\gamma_{n-1}+\gamma_{n-2}) ||Q_{n-2}||^2  \label{eq:rn-2}
\end{align}
and since $Q_n(x;z) = x^n + \mathcal{O}(x^{n-1})$ 
\begin{equation}
||Q_{n-4}||^2r_{n,n-4} = \bigl<\bm{v},x^4Q_nQ_{n-4}\bigr> = ||Q_n||^2.
\label{eq:rn-4}
\end{equation}
Using (\ref{eq:||Qn||/||Qn-1||}) in (\ref{eq:rn-2}) and (\ref{eq:rn-4}) we
get 
\begin{equation*}
\dot{Q}_n(x;z) = \frac{\gamma_n\gamma_{n-1}}{z^{3/2}}\biggl[%
\gamma_{n+1}+\gamma_n+\gamma_{n-1}+\gamma_{n-2}\biggr]Q_{n-2}(x;z)+\biggl[%
\frac{\gamma_n\gamma_{n-1}\gamma_{n-2}\gamma_{n-3}}{z^2}\biggr]Q_{n-4}(x;z).
\label{eq:Qpunto}
\end{equation*}
Thus, taking into account (\ref{eq:QnPn}) we obtain 
\begin{align*}
& -\frac{n}{4}z^{-n/4-1}P_n(z^{1/4}x;z) + z^{-n/4}\biggl[\frac{x}{4}%
z^{-3/2}\partial_yP_n(y;z)\biggr|_{y=z^{1/4}x} + \dot{P}_n(z^{1/4}x;z)\biggr]
\\
& = \frac{\gamma_n\gamma_{n-1}}{z^{3/2}}\left[\gamma_{n+1}+\gamma_n+%
\gamma_{n-1}+\gamma_{n-2}\right]z^{-\frac{n}{4}+\frac{1}{2}%
}P_{n-2}(z^{1/4}x;z) + \left[\frac{\gamma_n\gamma_{n-1}\gamma_{n-2}%
\gamma_{n-3}}{z^2}\right]z^{-\frac{n}{4}+1}P_{n-4}(z^{1/4}x;z).
\end{align*}
The result follows by multiplying both sides by $4z^{n/4+1}$ and changing $%
z^{1/4}x\longrightarrow x$.
\end{proof}

\begin{corollary}
Using (\ref{eq:3trrFreud}) we can reduce (\ref{eq:Pnpunto}) to 
\begin{align}
-&nP_n(x;z) + x\partial_xP_n(x;z) + z\dot{P}_n(x;z) +4\gamma_n(z)\biggl[%
\gamma_{n+1}(z)+\gamma_n(z)+\gamma_{n-1}(z)+x^2\biggr]P_n(x;z)  \notag \\
&= 4\gamma_n(z)\biggl[\gamma_{n+1}(z)+\gamma_n(z)+x^2\biggr]xP_{n-1}(x;z).
\label{eq:L1z}
\end{align}
Then the differential operator 
\begin{equation*}
\mathcal{L}_n = x\partial_x +z\partial_z+4\gamma_n\biggl[\gamma_{n+1}+%
\gamma_n+\gamma_{n-1}+x^2\biggr]-n.
\end{equation*}
satisfies 
\begin{equation*}
\mathcal{L}_nP_n(x;z) = 4\gamma_n\biggl[\gamma_{n+1}+\gamma_n+x^2\biggr]%
xP_{n-1}(x;z).
\end{equation*}
\end{corollary}

Now, we will deduce a differential equation which is satisfied by the zeros
of $P_n(x;z)$.

\begin{proposition}
Let $(x_{n,k})_{1\leq k\leq n}$ be the zeros of $P_n(x;z)$. Then they
satisfy the differential equation 
\begin{equation*}
\dot{x}_{n,k}(z)=\frac{x_{n,k}(z)}{z}\biggl[\frac{\mathcal{A}_n(z;z)}{%
\mathcal{A}_n(x_{n,k}(z);z)}\biggr],
\end{equation*}
where the polynomial $\mathcal{A}_n(x;z)$ was defined in (\ref{eq:AnFreud}).
\end{proposition}

\begin{proof}
Taking derivatives with respect to $z$ in (\ref{eq:Pn(xnk)})

\begin{equation*}
\partial_xP_n(x;z)\biggr|_{x=x_{n,k}}\dot{x}_{n,k}+\dot{P}_n(x_{n,k};z) = 0.
\end{equation*}
Then 
\begin{equation}
\dot{x}_{n,k} = -\left[\frac{\dot{P}_n(x;z)}{\partial_xP_n(x;z)}\right]%
_{x=x_{n,k}}.  \label{eq:xnkpunto}
\end{equation}
From (\ref{eq:L1Pn}) we have 
\begin{equation}
\partial_xP_n(x,z)\biggl|_{x=x_{n,k}} = \frac{4\gamma_n(z)\mathcal{A}%
_n(x_{n,k},z)}{\phi(x_{n,k},z)}P_{n-1}(x_{n,k},z).  \label{eq:L1Pnxnk}
\end{equation}
On the other hand, from (\ref{eq:L1z}) we get 
\begin{equation}
\dot{P}_n(x_{n,k},z) = \frac{x_{n,k}}{z}\biggl[4\gamma_n(\gamma_{n+1}+%
\gamma_n+x^2_{n,k})-\partial_xP_n(x_{n,k};z)\biggr]P_{n-1}(x_{n,k};z).
\label{eq:Pnpuntoxnk}
\end{equation}
Using (\ref{eq:L1Pnxnk}) and (\ref{eq:Pnpuntoxnk}) in (\ref{eq:xnkpunto}) we
obtain

\begin{equation*}
\dot{x}_{n,k} = \frac{x_{n,k}}{z\mathcal{A}_n(x_{n,k};z)}\biggl[\mathcal{A}%
_n(x_{n,k};z)-\phi(x_{n,k};z)(\gamma_{n+1}+\gamma_n+x^2_{n,k})\biggr].
\end{equation*}
From $(\ref{eq:AnFreud})$ our statement follows.
\end{proof}

\section{Numerical examples}

In this section, we derive and analyze the zeros of the symmetric truncated Freud polynomials $\{P_{n}(x;z)\}_{n\geq 0}$, along with other related quantities. The zeros of these polynomials exhibit different behavior depending on whether the degree $n$ is even or odd. Using Mathematica$^{\copyright}$ software, we illustrate the variation of these zeros as a function of the parameter $z$ for certain truncated Freud polynomials of fixed degree. Specifically, in Figure \ref{fig:Freud5}, we present the last zero of the truncated Freud polynomial $P_{5}(x;z)$, of degree five, across a range of values for $z$ from $z=0.2$ up to infinity. These zeros are presented numerically in Table \ref{tab:FreudZeros5}. Specifically, in Figure \ref{fig:Freud5}, we present the last zero of the truncated Freud polynomial $P_{5}(x;z)$, of degree five, across a range of values for $z$ from $z=0.2$ up to infinity. These zeros are presented numerically in Table \ref{tab:FreudZeros5}. 

Following this, Figure \ref{fig:Freud6} illustrates a similar situation for the degree-six polynomial $P_{6}(x;z)$, with $z$ values extending from close to zero to infinity, and the last two zeros of this polynomial are presented in Table \ref{tab:FreudZeros6}. It is worth noting that, at least in these low-degree polynomials, the zeros for $z\approx 2$ are already very close to the zeros of Freud polynomials, corresponding to $z\rightarrow \infty $. Additionally, the bottom panels of Figures \ref{fig:Freud5} and \ref{fig:Freud6} provide a zoomed-in view near the $x$-axis, allowing us to closely observe the graphs of the polynomials with smaller values of $z$, those with the dotted and dashed graphs, in greater detail. 

As in the previous sections, the zeros of the polynomial $P_{n}(x;z)$ are denoted by $x_{n,k}$. In the following tables, we present the last zeros of polynomials $P_{5}(x;z)$ and $P_{6}(x;z)$.

%%%%%%%%%%%%%%%%%%%%%%%%%%%%%%%%%%%%%%%%%%%%%%%%%%%%%%%%%%%%%%%%%%%%%%%%%%%%%%%%%%%%%%%%%%%%%%%%%%

%%%%%%%%%%%%%%%%%%%%%%%%%%%%%%%%%%%%%%%%%%%%%%%%%%%%%%%%%%%%%%%%%%%%%%%%%%%%%%%%%%%%%%%%%%%%%%%%%%

\begin{table}[!ht]
	\centering\renewcommand{\arraystretch}{1.2} 
	\begin{tabular}{|c|c|c|}
		\hline
		$z$ & $x_{5,4}$ & $x_{5,5}$ \\ \hline
		$0.2$ & 0.107685 & 0.181230 \\ \hline
		$0.4$ & 0.215105 & 0.362271 \\ \hline
		$0.6$ & 0.320950 & 0.542169 \\ \hline
		$0.9$ & 0.468896 & 0.803263 \\ \hline
		$1.2$ & 0.584567 & 1.029070 \\ \hline
		$1.4$ & 0.631505 & 1.130200 \\ \hline
		$1.5$ & 0.644491 & 1.158470 \\ \hline
		$\infty$ & 0.655248 & 1.180460 \\ \hline
	\end{tabular}%
	\caption{The two largest zeros of $P_{5}(x;z) $ for several values of $z$.}
	\label{tab:FreudZeros5}
\end{table}

%%%%%%%%%%%%%%%%%%%%%%%%%%%%%%%%%%%%%%%%%%%%%%%%%%%%%%%%%%%%%%%%%%%%%%%%%%%%%%%%%%%%%%%%%%%%%%%%%%

%%%%%%%%%%%%%%%%%%%%%%%%%%%%%%%%%%%%%%%%%%%%%%%%%%%%%%%%%%%%%%%%%%%%%%%%%%%%%%%%%%%%%%%%%%%%%%%%%%

\begin{figure}[!ht]
	\centering
	\includegraphics[width=0.7\textwidth]{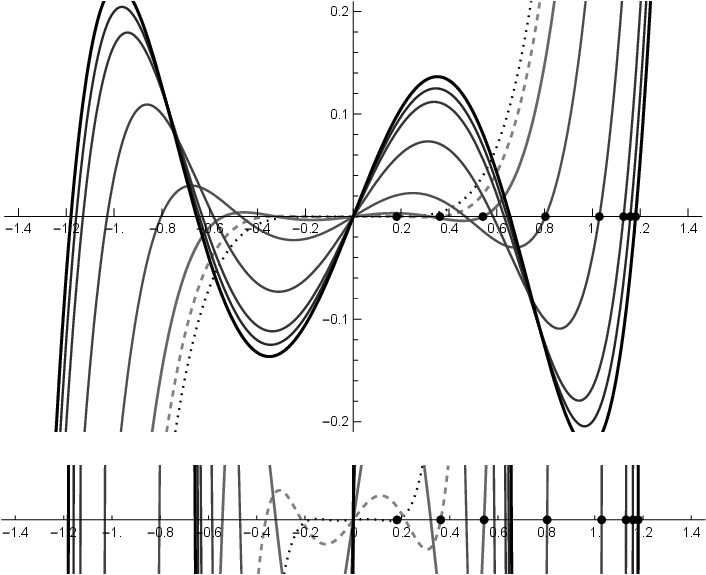}
	\caption{Behavior of the last zero of $P_{5}(x;z)$ for several values of $z$.}
	\label{fig:Freud5}
\end{figure}

%%%%%%%%%%%%%%%%%%%%%%%%%%%%%%%%%%%%%%%%%%%%%%%%%%%%%%%%%%%%%%%%%%%%%%%%%%%%%%%%%%%%%%%%%%%%%%%%%%

\bigskip

%%%%%%%%%%%%%%%%%%%%%%%%%%%%%%%%%%%%%%%%%%%%%%%%%%%%%%%%%%%%%%%%%%%%%%%%%%%%%%%%%%%%%%%%%%%%%%%%%%

\begin{table}[!ht]
	\centering\renewcommand{\arraystretch}{1.2} 
	\begin{tabular}{|c|c|c|}
		\hline
		$z$ & $x_{6,5}$ & $x_{6,6}$ \\ \hline
		$0.3$    & 0.1982974 & 0.2797096 \\ \hline
		$0.4$    & 0.2642083 & 0.3728558 \\ \hline
		$0.65$   & 0.4266802 & 0.6045856 \\ \hline
		$0.9$    & 0.5795190 & 0.8310874 \\ \hline 
		$1.2$    & 0.7308850 & 1.0794365 \\ \hline
		$1.4$    & 0.7984810 & 1.2066653 \\ \hline	
		$1.5$    & 0.8196970 & 1.2490734 \\ \hline
		$\infty$ & 0.8415723 & 1.2914650 \\ \hline
	\end{tabular}%
	\caption{The two largest zeros of $P_{6}(x;z)$ for various values of $z$.}
	\label{tab:FreudZeros6}
\end{table}

%%%%%%%%%%%%%%%%%%%%%%%%%%%%%%%%%%%%%%%%%%%%%%%%%%%%%%%%%%%%%%%%%%%%%%%%%%%%%%%%%%%%%%%%%%%%%%%%%%

%%%%%%%%%%%%%%%%%%%%%%%%%%%%%%%%%%%%%%%%%%%%%%%%%%%%%%%%%%%%%%%%%%%%%%%%%%%%%%%%%%%%%%%%%%%%%%%%%%

\begin{figure}[!ht]
	\centering
	\includegraphics[width=0.7\textwidth]{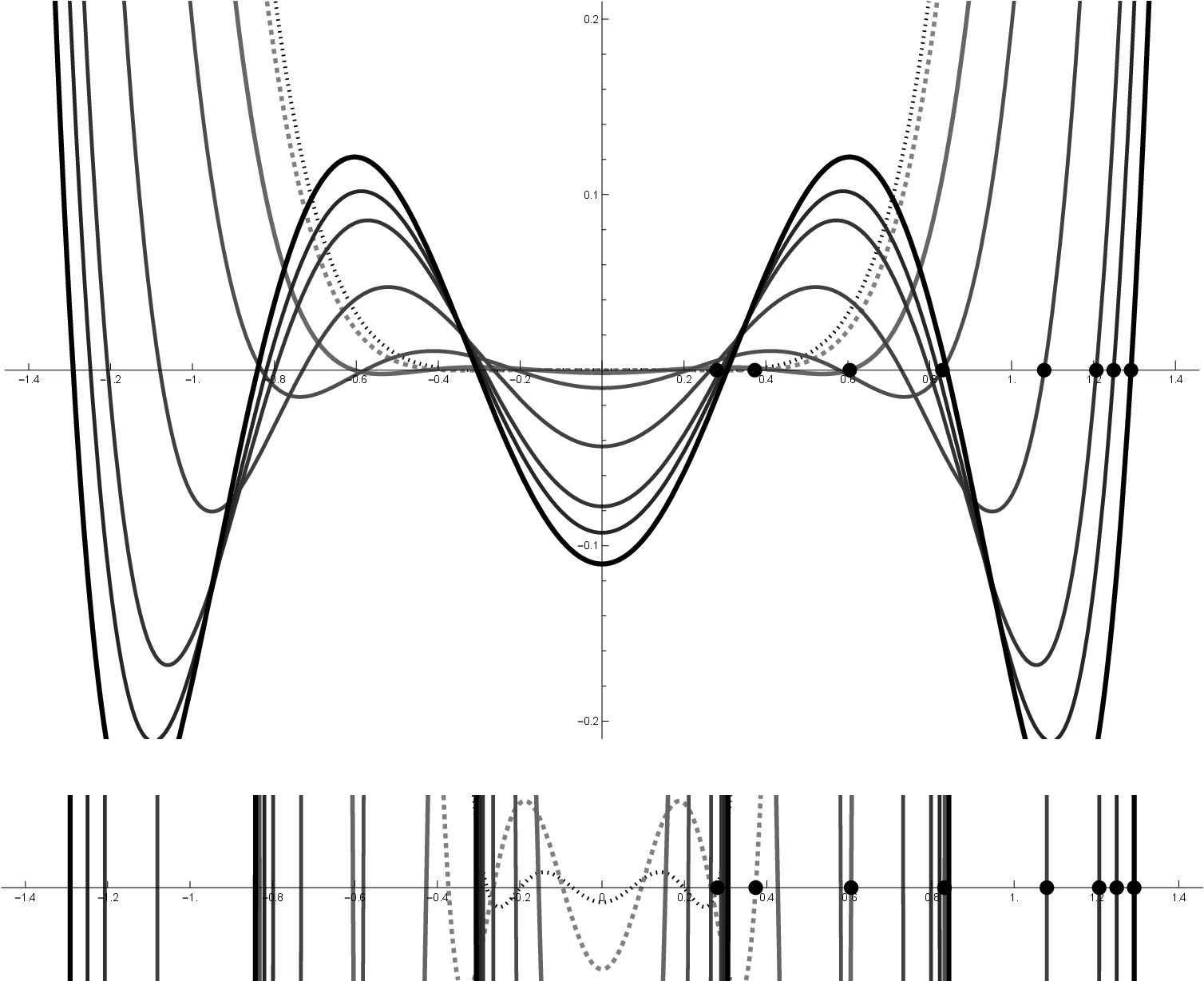}
	\caption{Behavior of the last zero of $P_{6}(x;z)$ for several values of $z$. }
	\label{fig:Freud6}
\end{figure}

%%%%%%%%%%%%%%%%%%%%%%%%%%%%%%%%%%%%%%%%%%%%%%%%%%%%%%%%%%%%%%%%%%%%%%%%%%%%%%%%%%%%%%%%%%%%%%%%%%

%%%%%%%%%%%%%%%%%%%%%%%%%%%%%%%%%%%%%%%%%%%%%%%%%%%%%%%%%%%%%%%%%%%%%%%%%%%%%%%%%%%%%%%%%%%%%%%%%%

In the following, we will study the zeros of a polynomial of interest due to its electrostatic interpretation, as developed in Section \ref%
{sect-Electr-Int}. Equation (\ref{eq:GE(xnk)}) there represents the electrostatic equilibrium condition for the $n$ real zeros $\{x_{n,k}\}_{k=1}^{n}$, of the truncated Freud polynomial $P_{n}(x;z)$, so therefore, the zeros of $P_{n}(x;z)$ represent the critical points of the total energy. In this way, the electrostatic interpretation of the distribution of the $n$ zeros, suggests an equilibrium configuration under the action of the external potential given in (\ref{totalExtPoten}), where the first term $V_{long}(x)$ is a long-range potential tied to a Christoffel perturbation of the Freud weight function, and $V_{n,short}(x;z)$ represents a short-range potential associated with unit charges located at the four zeros of the bi-quartic equation (\ref{bicuartica})%

\begin{equation*}
	\mathcal{A}_{n}(x;z)=x^{4}+b_{n}(z)x^{2}+c_{n}(z).
\end{equation*}%
We define in (\ref{etas})-(\ref{zetas}) the solutions of the above equation, namely
\begin{eqnarray*}
	\eta _{1,2}(n,z) &=&\pm \sqrt{\frac{1}{2}\left( \sqrt{b_{n}^{2}(z)-4c_{n}(z)}%
		-b_{n}(z)\right) }, \\
	\zeta _{1,2}(n,z) &=&\pm i\sqrt{\frac{1}{2}\left( \sqrt{%
			b_{n}^{2}(z)-4c_{n}(z)}+b_{n}(z)\right) }.
\end{eqnarray*}

Table \ref{Tabla3} shows the zeros of $\mathcal{A}_{n}(x;z)$ for a fixed values of $z=1$ and different increasing values of $n$. As $n$ increases, it is easy to see that $\mathcal{A}_{n}(x;z)$ exhibits two symmetric real zeros $\eta_{1}(n,z)$ and $\eta_{2}(n,z)$, along with two additional conjugate purely imaginary zeros $\zeta _{1}(n,z)$ and $\zeta _{2}(n,z)$. These zeros are represented in the complex plane in Figure \ref{fig:electzeros}.

%%%%%%%%%%%%%%%%%%%%%%%%%%%%%%%%%%%%%%%%%%%%%%%%%%%%%%%%%%%%%%%%%%%%%%%%%%%%%%%%%%%%%%%%%%%%%%%%%%

%%%%%%%%%%%%%%%%%%%%%%%%%%%%%%%%%%%%%%%%%%%%%%%%%%%%%%%%%%%%%%%%%%%%%%%%%%%%%%%%%%%%%%%%%%%%%%%%%%

\begin{figure}[!ht]
	\centerline{\includegraphics[width=9cm,keepaspectratio]{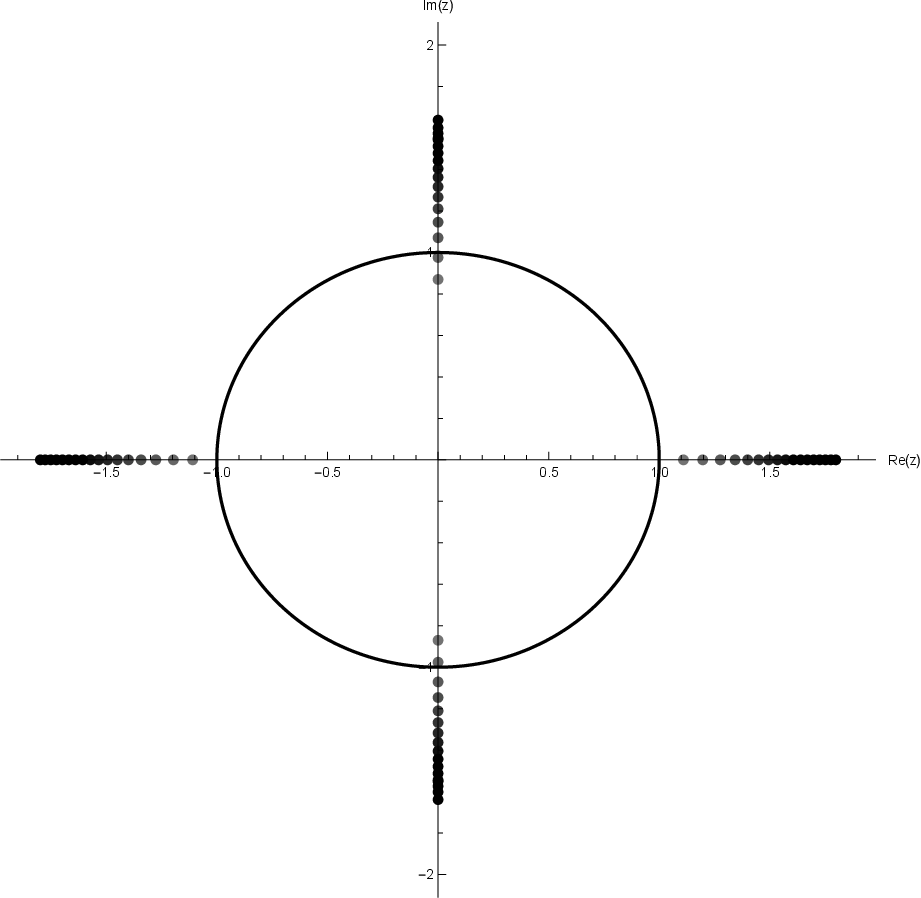}}
	\caption{Zeros of $\mathcal{A}_{n}(x;z)$ for $z=1$ and values of $n$ from $1$ to $17$.}
	\label{fig:electzeros}
\end{figure}

%%%%%%%%%%%%%%%%%%%%%%%%%%%%%%%%%%%%%%%%%%%%%%%%%%%%%%%%%%%%%%%%%%%%%%%%%%%%%%%%%%%%%%%%%%%%%%%%%%

%%%%%%%%%%%%%%%%%%%%%%%%%%%%%%%%%%%%%%%%%%%%%%%%%%%%%%%%%%%%%%%%%%%%%%%%%%%%%%%%%%%%%%%%%%%%%%%%%%

\begin{table}[!ht]
	\centering\renewcommand{\arraystretch}{1.2} {\small 
		\begin{tabular}{lcccc}
			\hline
%			& \multicolumn{2}{c}{$\pm \sqrt{z_{1}}$} & \multicolumn{2}{c}{$\pm \sqrt{z_{2}}$} \\ \cline{2-3}\cline{4-5}
			$n$ & $\eta_{1}(n,z)$ & $\zeta_{1}(n,z)$ & $\eta_{2}(n,z)$ & $\zeta_{2}(n,z)$ \\ \hline
			$n=1$ & $+ \,1.10947$ & $+ \,0.870003\,i$ & $- \,1.10947$ & $- \,0.870003\,i$ \\ 
			$n=2$ & $+ \,1.19659$ & $+ \,0.975701\,i$ & $- \,1.19659$ & $- \,0.975701\,i$ \\ 
			$n=3$ & $+ \,1.27583$ & $+ \,1.07093\,i$ & $- \,1.27583$ & $- \,1.07093\,i$ \\ 
			$n=4$ & $+ \,1.34277$ & $+ \,1.14612\,i$ & $- \,1.34277$ & $- \,1.14612\,i$ \\ 
			$n=5$ & $+ \,1.3997$ & $+ \,1.2106\,i$ & $- \,1.3997$ & $- \,1.2106\,i$ \\ 
			$n=6$ & $+ \,1.44951$ & $+ \,1.26696\,i$ & $- \,1.44951$ & $- \,1.26696\,i$ \\ 
			$n=7$ & $+ \,1.49408$ & $+ \,1.31726\,i$ & $- \,1.49408$ & $- \,1.31726\,i$ \\ 
			$n=8$ & $+ \,1.53462$ & $+ \,1.3628\,i$ & $- \,1.53462$ & $- \,1.3628\,i$ \\ 
			$n=9$ & $+ \,1.57192$ & $+ \,1.40449\,i$ & $- \,1.57192$ & $- \,1.40449\,i$ \\ 
			$n=10$ & $+ \,1.60653$ & $+ \,1.44302\,i$ & $- \,1.60653$ & $- \,1.44302\,i$ \\ 
			$n=11$ & $+ \,1.6389$ & $+ \,1.47888\,i$ & $- \,1.6389$ & $- \,1.47888\,i$ \\ 
			$n=12$ & $+ \,1.66932$ & $+ \,1.51246\,i$ & $- \,1.66932$ & $- \,1.51246\,i$ \\ 
			$n=13$ & $+ \,1.69806$ & $+ \,1.54408\,i$ & $- \,1.69806$ & $- \,1.54408\,i$ \\ 
			$n=14$ & $+ \,1.72533$ & $+ \,1.57404\,i$ & $- \,1.72533$ & $- \,1.57404\,i$ \\ 
			$n=15$ & $+ \,1.75127$ & $+ \,1.60162\,i$ & $- \,1.75127$ & $- \,1.60162\,i$ \\ 
			$n=16$ & $+ \,1.77603$ & $+ \,1.63848\,i$ & $- \,1.77603$ & $- \,1.63848\,i$ \\ 
			$n=17$ & $+ \,1.79819$ & $+ \,1.55118\,i$ & $- \,1.79819$ & $- \,1.55118\,i$ \\ \hline
		\end{tabular}%
	}
	\caption{Zeros $\protect\eta_{1,2}(n,z)$ and $\protect\zeta_{1,2}(n,z)$ for several values of $n$.}
	\label{Tabla3}
\end{table}

%%%%%%%%%%%%%%%%%%%%%%%%%%%%%%%%%%%%%%%%%%%%%%%%%%%%%%%%%%%%%%%%%%%%%%%%%%%%%%%%%%%%%%%%%%%%%%%%%%

%%%%%%%%%%%%%%%%%%%%%%%%%%%%%%%%%%%%%%%%%%%%%%%%%%%%%%%%%%%%%%%%%%%%%%%%%%%%%%%%%%%%%%%%%%%%%%%%%%

\section{Concluding remarks}

We have defined the truncated Freud polynomials $P_n(x;z)$, orthogonal with
respect to the linear functional 
\begin{equation*}
\bigl<\bm{u} , p(x) \bigr> = \displaystyle\int_{-z}^z p(x)e^{-x^4}dx, \quad
p\in \mathbb{P}, \quad z>0.
\end{equation*}
Such a linear functional satisfies the Pearson equation

\begin{equation*}
D((x^2-z^2)\bm{u})+(4x^3(x^2-z^2)-2x)\bm{u} = 0.
\end{equation*}
We have studied the MOPS $P_n(x;z)$ as semiclassical of class $4$. For this
sequence, we deduce the Laguerre-Freud equations, the first structure
relation, the lowering and raising operators as well as the corresponding
holonomic differential equation (in $x$) that it satisfies. As an
application, the dynamical behaviour and the electrostatic interpretation of
the zeros of such a sequence were derived. Finally, some illustrative
numerical tests concerning the zeros of the polynomials $P_n(x;z)$ are
given. In a further contribution we want to deal with the asymptotic
expansion of $P_n(x;z)$ as $n\rightarrow \infty$, $z\rightarrow \infty$ as
well as when both $n,z\rightarrow \infty$.

\subsection*{Acknowledgments}

The work of Francisco Marcell\'an has been supported by the research project
PID2021- 122154NB-I00 \emph{Ortogonalidad y Aproximaci\'on con Aplicaciones
en Machine Learning y Teor\'ia de la Probabilidad} funded by
MICIU-AEI-10.13039-501100011033 and by \emph{ERDF A Way of making Europe}.

The work of Edmundo J. Huertas and Alberto Lastra has been supported by the Ministerio de Ciencia e Innovaci\'on-Agencia Estatal de Investigaci\'on MCIN/AEI/10.13039/501100011033 and the European Union \textquotedblleft NextGenerationEU\textquotedblright /PRTR, under grant TED2021-129813A-I00. The work of Alberto Lastra is also partially supported by the project PID2019-105621GB-I00 of Ministerio de Ciencia e Innovaci\'on, Spain.

Part of this research was conducted while Edmundo J. Huertas was visiting the ICMAT (Instituto de Ciencias Matem\'aticas), from jan-2023 to jan-2024 under the Program \textit{Ayudas de Recualificaci\'on del Sistema Universitario Espa\~{n}ol para 2021-2023 (Convocatoria 2022) - R.D. 289/2021 de 20 de abril (BOE de 4 de junio de 2021)}. This author wish to thank the ICMAT, Universidad de Alcal\'a, and the
Plan de Recuperaci\'on, Transformaci\'on y Resiliencia (NextGenerationEU) of the
Spanish Government for their support.

%\newpage

% \section{The variable $z$}
% In this section we study $\bm{u}_{2n}$, $\gamma_{n}$ and the zeros of the polynomials $P_n(x;z)$ as functions of $z$. Setting $x=zt$ the moments read

% \begin{equation}
%     \bm{u}_{2n}(z) = \displaystyle\int_{-z}^zx^{2n}e^{-x^4}dx = z^{2n+1}\displaystyle\int_{-1}^1 t^{2n}e^{-z^4t^4}dx\label{eq:u'(z)}
% \end{equation}
% and we have the following result .

% \begin{proposition}
%     The sequence of even moments $\bm{u}_{2n}$ associated with the linear functional (\ref{eq:FreudLF}) satisfies the differential-difference relation
%     \begin{equation*}
%         \bm{u}'_{2n+2}=z^2\bm{u}'_{2n}
%     \end{equation*}

% \end{proposition}
% \begin{proof}
%     Taking the derivative with respect to $z$ in (\ref{eq:u'(z)}) gives
%     \begin{equation*}
%         z\bm{u}_{2n}' = (2n+1)\bm{u}_{2n}-4\bm{u}_{2n+4}.
%     \end{equation*}
%     Taking into account that
%     \begin{equation*}
%         \frac{d}{dz}\bigl(\phi x^{2n}\bigr)
%     \end{equation*}
% \end{proof}

\end{document}